\documentclass[11pt]{article}

\usepackage{geometry}
\usepackage[utf8]{inputenc}
\usepackage[english]{babel}
\usepackage[T1]{fontenc}
\usepackage{amsmath}
\usepackage{amsthm}
\usepackage{amsfonts}
\usepackage{amssymb}
\usepackage{lmodern}
\usepackage{verbatim}
\usepackage{ulem}

\usepackage[ruled]{algorithm}
\usepackage{algorithmic}

\usepackage{stmaryrd}

\usepackage{cancel}

\usepackage{hyperref}
\usepackage{url}

\usepackage{subfigure}
\usepackage{graphicx}

\usepackage{tikz}
\usepackage{tikz-3dplot}
\usetikzlibrary{3d}

\usetikzlibrary{positioning}
\usetikzlibrary{backgrounds}
\usetikzlibrary{patterns}
\usepackage{pgfplots}

\usepackage[font=small]{caption}

\usepackage{array}
\usepackage{graphbox}
\usepackage{xcolor,colortbl}

\usepackage{longtable}
\usepackage{booktabs}

\usepackage{wrapfig}
\usepackage{eurosym}
\usepackage{color}

\definecolor{Red}{rgb}{1,0,0}
\definecolor{Green}{rgb}{0,.6,0}
\definecolor{Blue}{rgb}{0,0,1}

\newcommand{\bl}{\color{Blue}}

\usepackage{pifont}

\theoremstyle{definition}
\newtheorem{definition}{Definition}[section]

\theoremstyle{remark}
\newtheorem*{remark}{Remark}

\theoremstyle{plain}

\theoremstyle{plain}
\newtheorem{theorem}{Theorem}[section]

\theoremstyle{plain}
\newtheorem{assumption}{Assumption}[section]

\providecommand{\keywords}[1]
{
	\small	
	\textbf{Key words.} #1
}

\title{
	{Handling of constraints in multiobjective blackbox optimization}
	\thanks{
		{GERAD}
		and D\'epartement de math\'ematiques et g\'enie industriel,
		Polytechnique Montr\'eal,
		C.P. 6079, Succ. Centre-ville,
		Montr\'eal, Qu\'ebec, Canada H3C~3A7.
		{This work is supported by
		the NSERC RGPIN-2018-05286 Discovery Grant and an IVADO fellowship.}   }
}

\author{
	\href{mailto:bigeon-j@univ-nantes.fr}{\bl Jean Bigeon}\thanks{
		\href{bigeon-j@univ-nantes.fr}{\bl \url{bigeon-j@univ-nantes.fr}},
		\href{https://www.univ-nantes.fr/jean-bigeon}{\bl \url{www.univ-nantes.fr/jean-bigeon}}.
	}
	\and
	\href{mailto:sebastien.le.digabel@gerad.ca}{\bl S\'ebastien Le~Digabel}\thanks{
	\href{mailto:sebastien.le.digabel@gerad.ca}{\bl \url{sebastien.le.digabel@gerad.ca}},
	\href{https://www.gerad.ca/Sebastien.Le.Digabel}{\bl \url{www.gerad.ca/Sebastien.Le.Digabel}}.
	}
	\and
	\href{mailto:ludovic.salomon@gerad.ca}{\bl Ludovic Salomon}\thanks{
	    \href{mailto:ludovic.salomon@gerad.ca}{\bl \url{ludovic.salomon@gerad.ca}},
		\href{https://www.gerad.ca/people/ludovic-salomon}{\bl \url{www.gerad.ca/Ludovic.Salomon}}.
	}
}

\date{}

\pgfplotsset{compat=1.17}

\begin{document}
	
	\maketitle
	
	\begin{abstract}
		This work proposes the integration of two new constraint-handling approaches into the blackbox constrained multiobjective optimization algorithm DMulti-MADS, an extension of the Mesh Adaptive Direct Search (MADS) algorithm for single-objective constrained optimization. The constraints are aggregated into a single constraint violation function which is used either in a two-phase approach, where research of a feasible point is prioritized if not available before improving the current solution set, or in a progressive barrier approach, where any trial point whose constraint violation function values are above a threshold are rejected. This threshold is progressively decreased along the iterations. As in the single-objective case, it is proved that these two variants generate feasible and/or infeasible sequences which converge either in the feasible case to a set of local Pareto optimal points or in the infeasible case to Clarke stationary points according to the constraint violation function. Computational experiments show that these two approaches are competitive with other state-of-the-art algorithms.
	\end{abstract}


	\section{Introduction}
	
	This work considers the following constrained multiobjective problem
	\begin{equation*}~\label{ref:MOP}
		MOP: \min_{x \in \Omega} f(x) = \left(f_1(x), f_2(x), \ldots, f_m(x)\right)^\top
	\end{equation*}
	where \(\Omega = \{x \in \mathcal{X} : c_j(x) \leq 0, \ j \in \mathcal{J}\} \subset \mathbb{R}^n\) is the \textit{feasible decision set}, and \(\mathcal{X}\) a subset of \(\mathbb{R}^n\). \(\mathbb{R}^n\) and \(\mathbb{R}^m\) are respectively designed as the \textit{decision space} and the \textit{objective space}. The functions \(f_i : \mathbb{R}^n \rightarrow \mathbb{R} \cup \{+ \infty\}\) for \(i = 1,2,\ldots, m\) and \(c_j : \mathbb{R}^n \rightarrow \mathbb{R} \cup \{+ \infty\}\) for \(j \in \mathcal{J}\) are the outputs of a program seen as a blackbox. In this context, no gradient is available nor cannot be approximated and one cannot make any assumption on the structure of the problem (differentiability, continuity, convexity) in absence of analytical expressions for the objective and the constraint functions. Many engineering applications, which involve several, costly and conflicting objectives, over a given set of constraints, fit into this framework (see for example~\cite{Alexandropoulos2019, DiPiKhuSaBe2009, Fang2005, Sharma2013}). For more general information on derivative-free methods, the reader is referred to~\cite{AuHa2017, CoScVibook, LaMeWi2019}.
	
	The description of the feasible decision set \(\Omega\) enables the modeller to distinguish different types of constraints~\cite{LedWild2015}. The set \(\mathcal{X}\) is the set of \textit{unrelaxable constraints}, which cannot be violated along the optimization process (e.g. strict bound constraints). The constraints in \(\Omega\) constitute the set of \textit{relaxable and quantifiable} \textit{constraints}, that can be violated during the optimization, i.e. the blackbox will execute and the constraints outputs can be aggregated as a measure of violation of the constraints. Finally, \textit{hidden constraints} constitute the set of points in the decision space for which the blackbox does not return any value, typically when the blackbox fails to execute. Allowing the \(f_i\) and \(c_j\) functions to take infinity values refers to this last type of constraints.
	
	Furthermore, in a multiobjective optimization context, due to the conflict between different objectives, a solution is not always optimal for all criteria. The goal is then to provide the set of best trade-off solutions to the decision maker~\cite{Branke2008Multiobjective,Collette2011Multiobjective, Miettinen_99_a}.

	In single-objective optimization, many algorithms have been proposed to solve blackbox constrained optimization problems: direct search methods via the use of a filter~\cite{AuDe09a, AuDe04a}, a merit function~\cite{GrVi2014} or an augmented Lagrangian~\cite{TGKolda_RMLewis_VTorczon_2006a}, a derivative-free linesearch algorithm coupled with a penalty function~\cite{LiLu2009}, or quadratic model-based approaches (see for example~\cite{AuCoLedPey2016,AuMa2014,DiEhMaPe2011}). The reader can refer to~\cite[Section~7]{LaMeWi2019} for a more thorough review.
	
	Evolutionary algorithms~\cite{Branke2008Multiobjective} are popular methods to tackle constrained multiobjective optimization blackbox problems. Firstly investigated in the context of bound-constrained or unconstrained blackbox optimization, researchers have adapted some of them to take into account inequality constraints (see~\cite{Yuan2021, Zhou2011} for more details). However, these methods are mostly stochastic heuristics. They practically require an important number of evaluations to perform. For example, the authors in~\cite{Yuan2021} suggest a budget comprised between \(2 \times 10^5\) and \(5 \times 10^5\) function evaluations in their experiments, which can be impracticable when evaluations are too costly. Surrogate models remove this limitation by substituting true blackboxes by less expansive surrogates, such as  radial basis functions (see~\cite{Regis2016}), or Gaussian processes (see~\cite{FeBeVa2017}).
	
	Recently, researchers have proposed extensions of convergence-based deterministic single-objective methods to multiobjective constrained optimization. Among the first ones, BiMADS~\cite{AuSaZg2008a,AuSaZg2008a} and MultiMADS~\cite{AuSaZg2010a} are scalarization-based algorithms. These two algorithms reformulate the multiobjective optimization problem into a succession of single-objective subproblems. Each of them is solved by the single-objective constrained blackbox MADS algorithm~\cite{AuDe2006, AuDe09a}. Practically, it can be difficult to correctly allocate the total budget of evaluations between all the subproblems, potentially resulting in a loss of evaluations required to improve the diversity and density of the current solution set.
	
	The Direct MultiSearch (DMS)~\cite{CuMaVaVi2010}, its variants~\cite{BraCu2020, CuMa2018, DedDesNa2021} and Derivative-Free MultiObjective (DFMO)~\cite{LiLuRi2016} algorithms consider a different approach. They all keep a list of current feasible best non-dominated solutions that they improve along the iterations. DMS extends single-objective direct search algorithms to constrained multiobjective optimization. It rejects non-feasible points via the use of an extreme barrier function approach, i.e. non feasible points are assigned infinity values. This approach does not exploit knowledge of constraint violations values, which could potentially help to improve the solution set. Furthermore, DMS imposes the use of a feasible starting point, which is not practically available (too costly) in a real engineering context. DFMO extends a derivative-free linesearch algorithm to constrained multiobjective optimization. By aggregating constraints with the objective functions via the use of a penalty function, DFMO reduces the initial constrained multiobjective optimization problem to a simple bound constrained multiobjective optimization problem, easier to solve. However, its convergence assumptions are a bit restrictive in a blackbox optimization context, i.e. constraints functions and objective functions must be Lipschitz continuous. On the contrary, DMS requires that objective functions should satisfy locally Lipschitz continuity. Besides, penalty function approaches can be sensitive to the scale of constraints (not always available in a blackbox context) and their penalty parameters values.
	
	Based on these remarks, this work proposes two other ways to handle blackbox constraints, based on the DMulti-MADS algorithm~\cite{BiLedSa2020,BiLedSa2020}. DMulti-MADS is an extension of the MADS algorithm to multiobjective optimization, strongly inspired by the DMS and BiMADS algorithms. It possesses stronger convergence properties than DMS. At the same time, experiments have shown its competitiveness according to state-of-the-art solvers on synthetic bound-constrained problems~\cite{BiLedSa2020}. Similarly to DMS, the first version of DMulti-MADS, described in~\cite{BiLedSa2020} requires a feasible starting point. The two extensions described below remove this limitation. The first one is an extension of the two-phase MADS algorithm described in~\cite{AuDeLe10} to constrained multiobjective optimization, named DMulti-MADS-TEB. The second version is an extension of the MADS algorithm with progressive barrier~\cite{AuDe09a} to constrained multiobjective optimization, designed as DMulti-MADS-PB. Contrary to a penalty function approach, this last method
	\begin{itemize}
		\item is less sensitive to the scale of the outputs of the blackbox as it does not aggregate the constraints with the objective function;
		\item allows to explore around several incumbent points and not just one;
		\item is proved to be have stronger convergence properties than DFMO.
	\end{itemize}
	
	This work is organized as follows. Section~\ref{sect:Pareto_dominance} provides a summary of multiobjective optimization concepts. Section~\ref{sect:Presentation_DMulti_MADS} introduces the core elements of the DMulti-MADS algorithm. Section~\ref{sect:handling_constraints_description} describes the DMulti-MADS-TEB and DMulti-MADS-PB variants to handle blackbox constraints. Main convergence results are detailed in Section~\ref{sect:Convergence_analysis}. Finally, experiments are conducted in Section~\ref{sect:Numerical_experiments} on synthetic benchmarks and three real engineering applications in comparison with other state-of-the-art solvers. 
	
	\section{Pareto dominance and optimal solutions in multiobjective optimization}\label{sect:Pareto_dominance}
	
	This section summarizes some notation and concepts of multiobjective optimization. In order to characterize optimal solutions, one needs the concept of \textit{Pareto dominance}~\cite{Miettinen_99_a}.
	
	\begin{definition}
		Given two feasible decision vectors \(x^1\) and \(x^2\) in \(\Omega\), 
		\begin{itemize}
			\item \(x^1 \preceq x^2\) (\(x^1\) \textit{weakly dominates} \(x^2\)) if and only if \(f_i(x^1) \leq f_i(x^2)\) for \(i = 1,2, \ldots, m\).
			\item \(x^1 \prec x^2\) (\(x^1\) \textit{dominates} \(x^2\)) if and only if \(f_i(x^1) \leq f_i(x^2)\) for \(i = 1,2, \ldots, m\) and there exists at least an index \(i_0 \in \{1,2, \ldots, m\}\) such that \(f_{i_0}(x^1) < f_{i_0}(x^2)\).
			\item \(x^1 \sim x^2\) (\(x^1\) and \(x^2\) are \textit{incomparable}) if and only if \(x^1\) does not dominate \(x^2\) and \(x^2\) does not dominate \(x^1\).
		\end{itemize} 
	\end{definition}

	With this definition, one is able to characterize locally optimal solutions and global optimal solutions in a multiobjective context.
	
	\begin{definition}
		A feasible decision vector \(x^\star \in \Omega\) is said to be (globally) \textit{Pareto optimal} if it does not exist any other decision vector \(x \in \Omega\) which dominates \(x^\star\).
	\end{definition}

	\begin{definition}
		A feasible decision vector \(x^\star \in \Omega\) is said to be \textit{locally Pareto optimal} if it exists a neighbourhood \(\mathcal{N}(x^\star)\) of \(x^\star\) such that there does not exist any other  decision vector \(x \in \mathcal{N}(x^\star) \cap \Omega\) which dominates \(x^\star\).
	\end{definition}

	The set of all Pareto optimal solutions in \(\Omega\) is called the \textit{Pareto set} denoted by \(\mathcal{X}_p\) and its image by the objective function is designed as the \textit{Pareto front} denoted by \(\mathcal{Y}_P \subseteq \mathbb{R}^m\). Any set of locally Pareto optimal solutions is called a \textit{local Pareto set}. Ideally, one would wish to find the entire Pareto set and consequently the entire Pareto front. But the Pareto set may be composed of an infinite number of solutions. In practice, an algorithm tries to find a representative set of nondominated points, denoted as a \textit{Pareto set approximation}~\cite{ZiKnTh2008} (its mapping by the objective function \(f\) is designed as a \textit{Pareto front approximation}). In the best case, a Pareto set approximation should be a subset of the Pareto set or a locally Pareto set, but this condition is not always satisfied.
	
	Several objective vectors, i.e. points in the objective space, play an important role in multiobjective optimization as bounds on the Pareto front.
	The \textit{ideal objective vector}~\cite{Miettinen_99_a} \(y^I \in \mathbb{R}^m\) bounds the Pareto front from below and is defined as
	\[y^I = \left(\min_{x \in \Omega} f_1(x), \min_{x \in \Omega} f_2(x), \ldots, \min_{x \in \Omega} f_m(x) \right)^\top.\]
	From each component of the ideal objective vector, one can obtain information on the \textit{extreme points of the Pareto set}, i.e. the elements of the Pareto set and solutions of each single-objective problem \(\min_{x \in \Omega} f_i(x)\) for \(i = 1,2, \ldots, m\).
	The \textit{nadir objective vector}~\cite{Miettinen_99_a} \(y^{\text{N}} \in \mathbb{R}^m\) provides an upper bound on the Pareto front. It is defined as
	\[y^{\text{N}} = \left(\max_{x \in \mathcal{X}_P} f_1(x), \max_{x \in \mathcal{X}_P} f_2(x), \max_{x \in \mathcal{X}_P} f_2(x), \ldots, \max_{x \in \mathcal{X}_P} f_m(x) \right)^\top.\]
	
	\section{The DMulti-MADS algorithm}\label{sect:Presentation_DMulti_MADS}
	
	DMulti-MADS~\cite{BiLedSa2020} is a direct search iterative method designed to solve constrained multiobjective blackbox optimization problems. It is an extension of the MADS~\cite{AuDe2006} algorithm to multiobjective optimization, strongly inspired by the DMS~\cite{CuMaVaVi2010} and BiMADS algorithms~\cite{AuSaZg2008a}. The notations and following definitions are taken from~\cite{AuDe09a, AuHa2017}.
	
	\begin{definition}
		At iteration \(k\), the \textit{set of feasible incumbent solutions} is defined as
		\[F^k = \left\{\arg \min_{x \in V^k} \left\{f(x): x \in \Omega\right\}\right\}\]
		where \(V^k \subset \mathcal{X}\) is the set of trial points which have been evaluated by the start of iteration \(k\).
	\end{definition}
	 
	Thus, all points in \(V^k\) satisfy the set of unrelaxable constraints \(\mathcal{X}\). \(V^0 \neq \emptyset \subset \mathcal{X}\) is then the set of starting points provided by the user. DMulti-MADS keeps an \textit{iterate list of best feasible incumbents} found until iteration \(k\), denoted as \(L^k_F\) and defined as
	\[L^k_F = \left\{ (x^l, \Delta^l) : x^l \in F^k \text{ and } \Delta^l > 0, \ l = 1,2, \ldots, |L^k_F|  \right\}\]
	where \(\Delta^l\) is the \textit{frame size parameter} associated to the \(l\)th non-dominated point \(x^l\) of the list \(L^k_F\). As \(L^k_F\) keeps only feasible non-dominated points, it is possible that \(|F^k| \neq |L^k_F|\).
	
	At the beginning of each iteration \(k\), DMulti-MADS selects an element \((x^k, \Delta^k)\) of the list \(L^k_F\) as the current \textit{feasible frame center}, and generates a finite number of new candidates. To ensure the convergence properties, all generated candidates during iteration \(k\) must belong to the mesh \(M^k\) defined by
	\[M^k = \bigcup_{x \in V^k} \{x + \delta^k Dz : z \in \mathbb{N}^{n_D}\} \subset \mathbb{R}^n\]
	where \(\delta^k > 0\) is the \textit{mesh size parameter}; \(D = G Z \in \mathbb{R}^{n \times n_D}\) is a matrix whose columns form a positive spanning set for \(\mathbb{R}^n\) (see~\cite[Chapter 6]{AuHa2017} or~\cite[Chapter 2]{CoScVibook}) for some non-singular matrix \(G \in \mathbb{R}^{n \times n}\) and some integer matrix \(Z \in \mathbb{Z}^{n \times n_D}\). Note that \(G, Z\) and \(D\) do not depend on the iteration indexes. Generally, \(G\) and \(Z\) are chosen such as \(G = I_n\) and \(Z = [I_n \ - I_n] = D\), with \(I_n\) the identity matrix of dimensions \(n \times n\). Furthermore, the current incumbent selection must satisfy at least the following condition:
	\[(x^k, \Delta^k) \in \left\{(x, \Delta) \in L^k_F: \tau^{w^+} \Delta^k_{\max} \leq \Delta \leq \Delta^k_{\max} \right\}\]
	where \(\tau \in \mathbb{Q} \cap (0,1)\) is the \textit{frame size adjustment parameter}, \(w^+ \in \mathbb{N}\) a fixed integer and \(\Delta^k_{\max}\) the \textit{maximum frame size parameter} at iteration \(k\) defined as 
	\[\Delta^k_{\max} =  \max_{(x, \Delta) \in L^k_F} \Delta.\]
	The mesh size parameter \(\delta^l\) and frame size parameter \(\Delta^l\) associated to the \(l\)th non-dominated point \(x^l\) of \(L^k_F\) are linked to each other such that \(0 < \delta^l \leq \Delta^l\). When a subsequence of one of them goes to \(0\), so does the other. Typically, the following relation \(\delta^l = \min \left\{ \Delta^l, (\Delta^l)^2 \right\}\) meets these requirements.
	
	Each iteration is decomposed into two steps: the \textit{search} and the \textit{poll}. The search is an optional and flexible step which enables the user to design any strategy as long as the proposed trial points belong to the mesh \(M^k\) and their number is finite. A common strategy is the use of surrogate models, proposed for example in~\cite{BraCu2020}. The finite set of points used in the search is denoted by \(S^k\).
	
	The poll is more rigorously defined, as the convergence analysis depends on it. The trial points involved in this step, named the poll set and denoted by \(P^k\), must satisfy some specific requirements. More precisely, the construction of \(P^k\) involves the use of the current incumbent \(x^k\) and its associated frame size \(\Delta^k\) and mesh size \(\delta^k\) parameters to obtain a positive spanning set \(\mathbb{D}^k_{\Delta}\). Each column of \(\mathbb{D}^k_{\Delta}\) must be a nonnegative integer combination of the directions in \(D\); the distance from the current incumbent \(x^k\) to a poll point must be bounded by a multiple of the frame size parameter \(\Delta^k\). Note that the relation between \(\delta^k\) and \(\Delta^k\) given above meets these requirements. Formally, \(P^k\) is described as
	\[P^k = \{x^k + \delta^k d : d \in \mathbb{D}^k_{\Delta}\} \subset M^k.\]
	All new candidates generated during the search and the poll are assigned a frame size parameter value larger or equal to the frame size parameter of the current feasible frame center.
	
	If a new generated candidate dominates the current feasible incumbent, the iteration is marked as a \textit{success}. Otherwise, it is a \textit{failure} and the frame size parameter (and so the mesh size parameter) of the current feasible frame center is decreased. The iteration can be opportunistic, meaning that as soon as it is successful, the remaining candidates (if they exist) are not evaluated. In all cases, the iterate list \(L^k_F\) is filtered to keep only best non-dominated feasible points found until the end of this iteration.
	
	More details can be found in~\cite{BiLedSa2020}.
		
	\section{Handling of constraints with DMulti-MADS}\label{sect:handling_constraints_description}
	
	This section details several strategies to handle constraints with DMulti-MADS. The set of quantifiable and relaxable constraints is given by \(\Omega = \left\{ x \in \mathcal{X} : c_j(x) \leq 0, \ j \in \mathcal{J} \right\}\). A relaxable constraint can be violated during the optimization and still returns meaningful outputs for the blackbox. A quantifiable constraint provides a measure of violation of feasibility. All other types of constraints (unrelaxable, hidden, non quantifiable), if present are added to \(\mathcal{X}\).
	
	\subsection{The constraint violation function}
	
	Exploiting constraints to guide the algorithm towards optimal solutions requires a way to quantify constraint violations. The strategies described below rely on the \textit{constraint violation function} \(h : \mathbb{R}^n \rightarrow \mathbb{R} \cup \{+ \infty\}\) used in~\cite{AuDe09a} and defined by
	\[h(x) = \begin{cases}
		\displaystyle\sum_{j \in \mathcal{J}} \left(\max \left\{c_j(x), 0 \right\}\right)^2 & \text{if } x \in \mathcal{X}, \\
		+\infty & \text{otherwise}.
	\end{cases}\]
	With this definition, \(x\) belongs to \(\Omega\) if and only if \(h(x) = 0\), and \(0 < h(x) < + \infty\) when \(x\) is infeasible but belongs to \(\mathcal{X} \setminus \Omega\). The use of a squared function instead of common \(\ell^1\) norm enables some conservation of first-order smoothness properties.
	
	\subsection{The extreme barrier (EB)}
	
	Similarly to DMS~\cite{CuMaVaVi2010}, the original version of the DMulti-MADS algorithm~\cite{BiLedSa2020} treats constraints via the use of an extreme barrier approach. It replaces the objective function \(f\) by
	\[f_{\Omega}(x) = \begin{cases}
		\left(+ \infty, + \infty, \ldots, + \infty\right)^\top & \text{if } x \notin \Omega, \\
		f(x) & \text{otherwise}.
	\end{cases} \]
	In other terms, all infeasible points are assigned an infinite value. This approach requires a feasible starting point, which is not always available in an engineering context. To allow the use of an infeasible starting point, this work proposes a \textit{Two-phase Extreme Barrier (TEB)} approach, in the continuation of~\cite{AuDeLe10}. When starting from an infeasible point, the new strategy, called DMulti-MADS-TEB, performs a single-objective minimization of the \(h\) constraint violation function using the MADS algorithm. As soon as a feasible point is found, DMulti-MADS-TEB moves to the second phase, which is the minimization of~\ref{ref:MOP} from the feasible point found in the first phase.
	
	Although this approach is simple,
	its performance has never been investigated in the context of deterministic multiobjective derivative-free optimization. It also shares some convergence properties with MADS and DMulti-MADS, summarized in Section~\ref{sect:Convergence_analysis}. Note that this strategy can be applied to any multiobjective blackbox algorithm.
	
	\subsection{The progressive barrier (PB)}
	
	This subsection introduces the DMulti-MADS-PB extension of the single-objective MADS-PB algorithm~\cite{AuDe09a} for multiobjective derivative-free optimization.
	
	\subsubsection{Feasible and infeasible incumbents}
	
	Similarly to the MADS-PB algorithm~\cite{AuDe09a}, DMulti-MADS-PB constructs two sets of incumbent solutions from \(V^k\). \(F^k\) still denotes the set of feasible incumbent solutions. To define the set of infeasible incumbent solutions, one needs to extend the notion of dominance for infeasible solutions, as it is required in the design of filter algorithms~\cite{FlLe02a, FlLeTo02a}.
	
	\begin{definition}[Dominance relation for constrained multiobjective optimization]
	In the context of constrained multiobjective optimization, \(x^1 \in \mathcal{X}\) is said to dominate \(x^2 \in \mathcal{X}\) if
	\begin{itemize}
		\item Both points are feasible and \(x^1 \in \Omega\) dominates \(x^2 \in \Omega\), denoted as \(x^1 \prec_f x^2\).
		\item Both points are infeasible and \(f_i(x^1) \leq f_i(x^2)\) for \(i = 1,2, \ldots, m\) and \(h(x^1) \leq h(x^2)\) with at least one strict inequality, denoted as \(x^1 \prec_h x^2\).
	\end{itemize}
	\end{definition}

	This extension of the dominance relation is different from the definition proposed in~\cite{Regis2016}. Indeed, in this work, feasible and infeasible points are never compared, and the dominance relation takes into account both objective function values and the constraint violation function values. Another extension of dominance to constrained optimization appears in~\cite{FliVaz2016}, but as in the previous case, it allows the comparison of feasible and infeasible points. Note that if \(m = 1\), the dominance relation reduces into the dominance relation of MADS-PB~\cite{AuDe09a}.
	
	With this dominance relation, one can define the set of infeasible nondominated points.

	\begin{definition}
		At iteration \(k\), the \textit{set of infeasible nondominated points} is defined as
		\[U^k = \left\{x \in V^k \setminus \Omega : \text{ there is no } y \text{ such that } y \prec_h x  \right\}.\]
	\end{definition}

	As for the MADS-PB algorithm, DMulti-MADS-PB relies on a nonnegative barrier threshold \(h^k_{\max}\), set at each iteration \(k\), to construct the set of infeasible incumbent solutions.
	
	\begin{definition}
		At iteration \(k\), the \textit{set of infeasible incumbent solutions} is
		\[I^k = \left\{\arg \min_{x \in U^k} \left\{f(x) : 0 < h(x) \leq h^k_{\max} \right\}\right\}.\]
	\end{definition}
	
	All evaluated points having a value of \(h\) above \(h^k_{\max}\) are automatically rejected by the algorithm. Furthermore, the barrier threshold is nonincreasing with the iteration number \(k\). Its value at each iteration is detailed in Section~\ref{subsect:update_end_iteration}.
		
	Figure~\ref{fig:feasible_and_infeasible_incumbents_example} illustrates these definitions. Note that \(I^k\) is not a singleton. The images of fourteen trial points generated at the beginning of iteration \(k\), i.e. \(V^k\), in the ``augmented'' objective space (a triobjective space with two objectives \(f_1\), \(f_2\) and the constraint violation function \(h\)) for a biobjective minimization optimization problem are represented. The set of feasible incumbent solutions, indicated by black bullets, contains four elements. Two other generated points are equally visible, but each of them is dominated by a feasible incumbent solution. These six generated trial points belong to the biobjective space. The set of infeasible non-dominated points contains six elements, identified by black lozenges and diamonds. Among them, only three qualify to be infeasible incumbent solutions. Indeed, one element among the others is above the threshold value \(h^k_{\max}\). The two other ones are dominated by at least one solution of \(I^k\) in terms of \(f\) objective values. Two elements of \(U^k\) dominate the two last remaining trial points, marked by \(\times\) symbols. Notice that all elements among \(I^k\) and \(F^k\) could have been generated before iteration \(k-1\), by definition of \(V^k\).
	
	\begin{figure}[!th]
		\centering
		\tdplotsetmaincoords{45}{135}
		\begin{tikzpicture}
			
			\begin{scope}[xshift=-5.5cm, yshift=3.2cm, tdplot_main_coords]
				\draw [->] (0,0,0) -- (5,0,0) node[left]{\small{\(f_1\)}};
				\draw [->] (0,0,0) -- (0,5,0) node[right]{\small{\(f_2\)}};
				\draw [->] (0,0,0) -- (0,0,4.5) node[above right]{\small{\(h\)}};
				
				\fill [pattern=dots, pattern color=gray!70] (0,0,4) -- (5,0,4) -- (5, 5, 4) --  (0,5,4) --cycle;
				\draw (-0.3,2.5, 4) node[right]{\small{\(h^k_{\max}\)}};
				
				\draw (0.5, 4, 0) node{\tiny{\(\bullet\)}};
				\draw (1.5, 3, 0) node{\tiny{\(\bullet\)}};
				\draw (2, 1, 0) node{\tiny{\(\bullet\)}};
				\draw (3, 0.5, 0) node{\tiny{\(\bullet\)}};
				
				\draw (1.6, 3.8, 0) node{\(\circ\)};
				\draw (2.5, 1.3, 0) node{\(\circ\)};
				
				\draw[thick,dotted] (0.2, 4.5, 0) -- (0.2, 4.5, 3) node{\tiny{\(\blacklozenge\)}};
				\draw[thick,dotted] (1, 3.1, 0) -- (1, 3.1, 3.5) node{\tiny{\(\blacklozenge\)}};
				\draw[thick,dotted] (2.5, 0.75, 0) -- (2.5, 0.75, 2) node{\tiny{\(\blacklozenge\)}};
				
				\draw[thick,dotted] (1, 0.5, 0) -- (1, 0.5, 6) node{\(\diamond\)};
				\draw[thick,dotted] (1.5, 4.8, 0) -- (1.5, 4.8, 1.5) node{\(\diamond\)};
				\draw[thick,dotted] (4, 0.85, 0) -- (4, 0.85, 1) node{\(\diamond\)};
				
				\draw[thick,dotted] (4.5, 1.7, 0) -- (4.5, 1.7, 4) node{\(\times\)};
				\draw[thick,dotted] (3.7, 4.1, 0) -- (3.7, 4.1, 4.5) node{\(\times\)};
				
			\end{scope}
		
			
			\draw [->] (0,0) -- (5.5,0) node[below right]{\(f_1\)};
			\draw [->] (0,0) -- (0,6) node[above left]{\(f_2\)};
			
			\draw (0.5, 4) node{\tiny{\(\bullet\)}};
			\draw (1.5, 3) node{\tiny{\(\bullet\)}};
			\draw (2, 1) node{\tiny{\(\bullet\)}};
			\draw (3, 0.5) node{\tiny{\(\bullet\)}};
			
			\draw (1.6, 3.8) node{\(\circ\)};
			\draw (2.5, 1.3) node{\(\circ\)};
			
			\draw (0.2, 4.5) node{\tiny{\(\blacklozenge\)}};
			\draw (1, 3.1) node{\tiny{\(\blacklozenge\)}};
			\draw (2.5, 0.75) node{\tiny{\(\blacklozenge\)}};
			
			\draw (1, 0.5) node{\(\diamond\)};
			\draw (1.5, 4.8) node{\(\diamond\)};
			\draw[thick,dotted] (4, 0.85) node{\(\diamond\)};
			
			\draw (4.5, 1.7) node{\(\times\)};
			\draw (3.7, 4.1) node{\(\times\)};
			
			\draw (-10, -1.5) node{\tiny{\(\bullet\)}} node[right]{\small{Set of feasible non dominated points \(F^k\)}};
			\draw (-10, -2) node{\tiny{\(\blacklozenge\)}} node[right]{\small{Set of infeasible incumbent solutions \(I^k\)}};
			\draw (-10, -2.5) node{\(\diamond\)} node[right]{\small{Subset of the set of infeasible non-dominated points \(U^k \setminus I^k\)}};
			\draw (0, -1.5) node{\(\circ\)} node[right]{\small{Feasible dominated points}};
			\draw (0, -2) node{\(\times\)} node[right]{\small{Infeasible dominated points}};
			
		\end{tikzpicture}
		\caption{An example of feasible and infeasible incumbent solutions at iteration \(k\) for a biobjective minimization problem in the ``augmented'' objective space (a triobjective space with the two objectives \(f_1\), \(f_2\) and the constraint violation function \(h\)). On the left, a 3D view; on the right, the projection on the biobjective space.}
		\label{fig:feasible_and_infeasible_incumbents_example}
	\end{figure}
	
	From the sets \(F^k\) and \(I^k\), DMulti-MADS constructs two lists of incumbent solutions, the iterate list of best feasible incumbents found until iteration \(k\),
	\[L^k_F = \left\{ (x^l, \Delta^l) : x^l \in F^k \text{ and } \Delta^l > 0, \ l = 1,2, \ldots, |L^k_F| \right\}\]
	and the \textit{iterate list of best infeasible incumbents} found until iteration \(k\)
	\[L^k_I = \left\{ (x^l, \Delta^l) : x^l \in I^k \text{ and } \Delta^l > 0, \ l = 1,2, \ldots, |L^k_I| \right\}.\]
	Each element of both lists possesses its own associated frame size parameter \(\Delta^l\).
	
	\subsubsection{An iteration of the DMulti-MADS-PB algorithm}\label{subsect:iteration_dmulti-mads-pb}
	
	As for the single-objective optimization MADS-PB algorithm~\cite{AuDe09a}, the search and the poll which constitute the two steps of an iteration for DMulti-MADS-PB are organized around two iterate incumbents at iteration \(k\): a feasible one \((x^k_F, \Delta^k_F) \in L^k_F\) and an infeasible one \((x^k_I, \Delta^k_I) \in L^k_I\). However, as the frame size parameters associated to the feasible incumbent \(x^k_F\) and infeasible incumbent \(x^k_I\) can be distinct, it is necessary to adapt the definition of the mesh \(M^k\). At iteration \(k\), \(M^k\) is defined as
	\[M^k = \begin{cases}
	 \displaystyle\bigcup_{x \in V^k} \{x + \delta^k_F D z : z \in \mathbb{N}^{n_D}\} & \text{if } L^k_F \neq \emptyset; \\
	 \displaystyle\bigcup_{x \in V^k} \{x + \delta^k_I D z : z \in \mathbb{N}^{n_D}\} & \text{otherwise,}
	 \end{cases}\]
 	where \(\delta^k_F > 0\) and \(\delta^k_I > 0\) are respectively the mesh size parameters associated to the feasible and infeasible incumbents \(x^k_F\) and \(x^k_I\) defined as \(\delta^k_F = \min \left\{ \Delta^k_F, (\Delta^k_F)^2 \right\}\) and \(\delta^k_I = \min \left\{ \Delta^k_I, (\Delta^k_I)^2 \right\}\). In other terms, the configuration of the mesh \(M^k\) at iteration \(k\) is primarily based on the selection of the feasible frame center if this last one exists. 
 	
 	It is then possible to adapt the definition of the poll set \(P^k\). At iteration \(k\), \(P^k\) is defined as
 	\[P^k = \begin{cases}
 		P^k(x^k_F, \Delta^k_F) & \text{ for some } (x^k_F, \Delta^k_F) \in L^k_F \text{ if } L^k_I = \emptyset, \\
 		P^k(x^k_I, \Delta^k_I) & \text{ for some } (x^k_I, \Delta^k_I) \in L^k_I \text{ if } L^k_F = \emptyset, \\
 		P^k(x^k_F, \Delta^k_F) \cup P^k(x^k_I, \Delta^k_F) & \text{ for some } (x^k_F, \Delta^k_F) \in L^k_F \text{ and } (x^k_I, \Delta^k_I) \in L^k_I, \text{ otherwise},
 	\end{cases}\]
    where \(P^k(x, \Delta^k) = \{x + \delta^k d : d \in \mathbb{D}^k_\Delta\} \subset M^k\) represents the poll set centered at \(x\) at iteration \(k\) with \(\delta^k = \min \left\{ \Delta^k, (\Delta^k)^2 \right\}\). 
    
    Figure~\ref{fig:poll_set_example} illustrates a construction of the poll set \(P^k\) when both the feasible frame center \(x^k_F\) and the infeasible frame center \(x^k_I\) exist. Here, \(\Omega \subset \mathbb{R}^2\). All poll candidates belong to one of the frames generated by \(x^k_F\) or \(x^k_I\) of size \(\Delta^k_F > 0\) (this is not mandatory as long as the definition of the poll holds). The set \(P^k\) is the union of the sets \(P^k(x^k_F, \Delta^k_F) = \left\{p^1, p^2, p^3, p^4\right\}\) and \(P^k(x^k_I, \Delta^k_I) = \left\{p^5, p^6\right\}\).     Section~\ref{subsect:frame_center_selection} gives more implementation details on the construction of the poll set.
    
    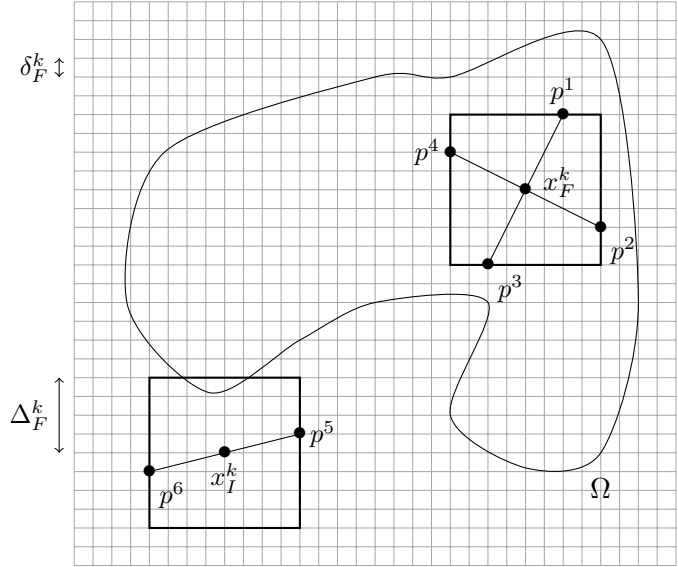
\begin{figure}[!ht]
    	\centering
    	\begin{tikzpicture}
    		
    		\draw[very thin, gray!70] (0, -0.5) grid[step=0.25] (8,7);
    		
    		\draw plot [smooth cycle] coordinates{(1.8,1.8) (0.7, 3) (1.2, 5) (4, 6) (5, 6) (7, 6.5) (7.5, 3) (7, 1) (6, 0.8) (5, 1.5) (5.5, 3) (4, 3) (3, 2.5)};
    		
    		\draw (6,4.5) node{\(\bullet\)};
    		\draw (6.1, 4.6) node[right]{\small{\(x^k_F\)}};
    		\draw[thick] (5, 3.5) rectangle(7, 5.5);
    		
    		\draw (6, 4.5) -- (6.5, 5.5) node{\(\bullet\)} node[above]{\small{\(p^1\)}};
    		\draw (6, 4.5) -- (7, 4) node{\(\bullet\)} node[below right]{\small{\(p^2\)}};
    		\draw (6, 4.5) -- (5.5, 3.5) node{\(\bullet\)} node[below right]{\small{\(p^3\)}};
    		\draw (6, 4.5) -- (5, 5) node{\(\bullet\)} node[left]{\small{\(p^4\)}};
    		
    		\draw (2, 1) node{\(\bullet\)} node[below]{\small{\(x^k_I\)}};
    		\draw[thick] (1, 0) rectangle (3, 2);
    		
    		\draw (2, 1) -- (3, 1.25) node{\(\bullet\)} node[right]{\small{\(p^5\)}};
    		\draw (2, 1) -- (1, 0.75) node{\(\bullet\)} node[below right]{\small{\(p^6\)}};
    		
    		\draw [<->] (-0.2, 1) -- (-0.2, 2) node[midway, left]{\small{\(\Delta^k_F\)}};
    		\draw [<->] (-0.2, 6) -- (-0.2, 6.25) node[midway, left]{\small{\(\delta^k_F\)}};
    		\draw (7,0.8) node[below]{\(\Omega\)};
    		
    	\end{tikzpicture}
    	 \caption{Example of a poll set \(P^k = \left\{p^1, p^2, p^3, p^4, p^5, p^6\right\}\) for \(\Omega \subset \mathbb{R}^2\) when both \(x^k_F\) and \(x^k_I\) exist (inspired by~\cite{AuDe09a}).}
    	 \label{fig:poll_set_example}
    \end{figure}
        
    It remains to address the choice of the feasible and infeasible frame centers at iteration \(k\). In the case of the MADS-PB algorithm, the set of feasible and infeasible incumbent solutions are often singletons (or composed of points which have the same objective function value and the same \(h\)-constrained value). Their selection is then unambiguous.
    
    When \(L^k_F\) possesses at least one element, the choice of the current feasible frame center must satisfy the same condition as described 
    in Section~\ref{sect:Presentation_DMulti_MADS}. Practically, to get a good Pareto front approximation, it is also recommended to take into consideration the gap between the different non-dominated solutions found until iteration \(k\), as it is done in~\cite{BiLedSa2020}.
    
    If \(L^k_F\) is empty, the infeasible frame center must satisfy
    \[(x^k_I, \Delta^k_I) \in \left\{(x, \Delta) \in L^k_I : \Delta^k_{h_{\min}} \leq \Delta \right\}\]
    where \(\Delta^k_{h_{\min}}\) is defined as
    \[(x^k_{h_{\min}}, \Delta^k_{h_{\min}}) \in \arg \min_{(x, \Delta) \in L^k_I} h(x).\]
    The idea behind this selection criterion is to prioritize exploration along the least infeasible point with the best objective values hoping to find a ``good'' feasible point. At the same time, this selection criterion allows to explore some potentially interesting regions of the objective space. Intuitively, the infeasible current best incumbents associated with a frame size parameter value superior to the least current infeasible point are the ones which have not yet been explored or are promising, due to the update procedure, detailed in Section~\ref{subsect:update_end_iteration}. Several infeasible incumbents can satisfy this criterion. Then, following the selection procedure described in~\cite{BiLedSa2020}, this work proposes Algorithm~\ref{alg:select_point_strategy_infeasible_nofeasible} to take into account the density of the set of best infeasible incumbents in the objective space.
    
    \begin{figure}[!th]
    	\begin{algorithm}[H]\small
    		\caption{\texttt{selectCurrentInfeasibleIncumbent}\((L^k_I)\)}
    		\begin{algorithmic}
    			\STATE Let \(L^{select}_I := \left\{(x, \Delta) \in L^k_I: \Delta^k_{h_{\min}} \leq \Delta\right\}\) with \(\Delta^k_{h_{\min}} = \arg \displaystyle\min_{(x, \Delta) \in L^k_I} h(x)\).
    			\IF{\(|L^{select}_I| = 1 \)}
    			\RETURN \((x, \Delta)\) with \(L^{select}_I = \left\{(x, \Delta)\right\}\).
    			\ELSIF{\(|L^{select}_I| = 2\) and  \(|L^k_I| = 2\) }
    			\STATE Let \(l_0 \in \underset{l = 1,2}{\arg\max} \ \max_{i = 1, 2, \ldots, m} f_i (x^l)  \).
    			\RETURN \((x^{l_0}, \Delta^{l_0}) \).
    			\ELSE
    			\FOR{\(i = 1,2, \ldots, m\)}
    				\STATE Order \(L^k_I = \left\{(x^1, \Delta^1), (x^2, \Delta^2), \ldots, (x^{|L^k_I|}, \Delta^{|L^k_I|}) \right\}\) such that \\ \(f_i(x^1) \leq f_i(x^2) \leq \ldots \leq f_i(x^{|L^k|})\).
    				\FOR{\(l = 1,2, \ldots, |L^k_I|\)}
    					\STATE Compute \(\gamma_i(x^l)\) defined as
    					\[\gamma_i(x^l) = \begin{cases}
    						2 \dfrac{f_i(x^2) - f_i(x^1)}{f_i(x^{|L^k_I|}) - f_i(x^1)} & \text{if } l = 1, \\
    						2 \dfrac{f_i(x^{|L^k_I|}) - f_i(x^{|L^k_I| - 1})}{f_i(x^{|L^k_I|}) - f_i(x^1)} & \text{ if } l = |L^k_I|,\\
    						\dfrac{f_i(x^{l+1}) - f_i(x^{l-1})}{f_i(x^{|L^k_I|}) - f_i(x^1)} & \text{otherwise}.
    					\end{cases}\]
    				\ENDFOR
    			\ENDFOR 
    			\STATE Let \(l_0 \in \underset{l = 1, 2, \ldots , |L^{select}_I|}{\arg\max} \max_{i = 1, 2, \ldots, m} \gamma_i (x^l) \).
    			\RETURN \((x^{l_0}, \Delta^{l_0}) \).
    			\ENDIF
    		\end{algorithmic}
    		\label{alg:select_point_strategy_infeasible_nofeasible}
    	\end{algorithm}
    	\caption{A procedure to select the current incumbent at iteration \(k\) taking into account the spacing between elements of the iterate list of best infeasible incumbents \(L^k_I\) in the objective space, inspired by~\cite{BiLedSa2020}.}
    \end{figure}
    
    There remains the case where \(L^k_F\) and \(L^k_I\) are both non-empty. A first approach would be to independently select the feasible and infeasible frame centers, based for example on a spacing criterion to densify the set of best feasible and best infeasible current solutions. However, this strategy does not exploit the ``dominance'' order which exists between both sets. More precisely, one could hope that exploring around a carefully chosen infeasible incumbent leads to the generation of a new feasible point which significantly improves the set of current feasible solutions. The proposed approach is inspired by the works of~\cite{Li2017}.

    At iteration \(k\), considering the non-empty iterate list of feasible incumbents \(L^k_F\), this work introduces the function \(\psi_{L^k_F} : \mathcal{X} \rightarrow \mathbb{R}\) given as
    \begin{align*} 
    \psi_{L^k_F}(x) & = \Phi_{L^k_F}(f(x)) \\
    & =
     \begin{cases}
    	\displaystyle\min_{(x^F, \Delta) \in L^k_F} \sum_{i = 1}^m \left[f_i(x^F) - \min \left\{f_i(x), f_i(x^F)\right\}\right] & \text{if there is no } (x^F, \Delta) \in L^k_F \text{ such that } \\ 
    	& f_i(x^F) \leq f_i(x) \text{ for } i = 1,2,\ldots,m; \\
    	-\displaystyle\min_{(x^F, \Delta) \in L^k_F} \sum_{i = 1}^m \left[f_i(x) - \min \left\{f_i(x), f_i(x^F)\right\}\right] & \text{otherwise}.
    \end{cases}
 	\end{align*}
 	The level sets of \(\Phi_{L^k_F}\) are depicted in Figure~\ref{fig:level_sets_phi_fct}. Note that all potential feasible decision vectors which are not dominated by a current feasible incumbent solution of \(L^k_F\) are given a positive \(\psi_{L^k_F}\) value. All dominated feasible decision vectors correspond to a negative \(\psi_{L^k_F}\) value.
 	
 		\begin{figure}[!th]
 		\centering
 		\begin{tikzpicture}
 			\draw [->] (0,0) -- (7, 0);
 			\draw [->] (0,0) -- (0, 5);
 			\draw (7, 0) node[below right] {{\small $f_1$}};
 			\draw (0, 5) node[above left] {{\small $f_2$}};
 			
 			\draw (1, 4) node{{\tiny \(\bullet\)}};
 			\draw (2, 3) node{{\tiny \(\bullet\)}};
 			\draw (3.5, 2.5) node{{\tiny \(\bullet\)}};
 			\draw (4.5, 2) node{{\tiny \(\bullet\)}};
 			\draw (5, 0.7) node{{\tiny \(\bullet\)}};
 			
 			\draw (1, 5) -- (1, 4) -- (2, 4) -- (2, 3) -- (3.5, 3) -- (3.5,2.5) -- (4.5, 2.5) -- (4.5, 2) -- (5, 2) -- (5, -0.2);
 			
 			\draw (1.5, 1.5) node{\small{\(\Phi_{L^k_F} > 0\)}};
 			\draw (5.5, 4) node{\small{\(\Phi_{L^k_F} < 0\)}};
 			\draw (2,3) -- (3.5, 3)  node[midway, below]{\small{\(\Phi_{L^k_F} = 0\)}};
 			
 			\draw (8, 3.5) node{{\tiny \(\bullet\)}} node[right] {\small{Points in \(L^k_F\)}};
 			\draw (7.8, 3.2) rectangle (10.3, 3.8);
 			
 		\end{tikzpicture}
 		\caption{Level sets in the objective space of \(\Phi_{L^k_F}\) for a biobjective minimization problem.}
 		\label{fig:level_sets_phi_fct}	
 	\end{figure}
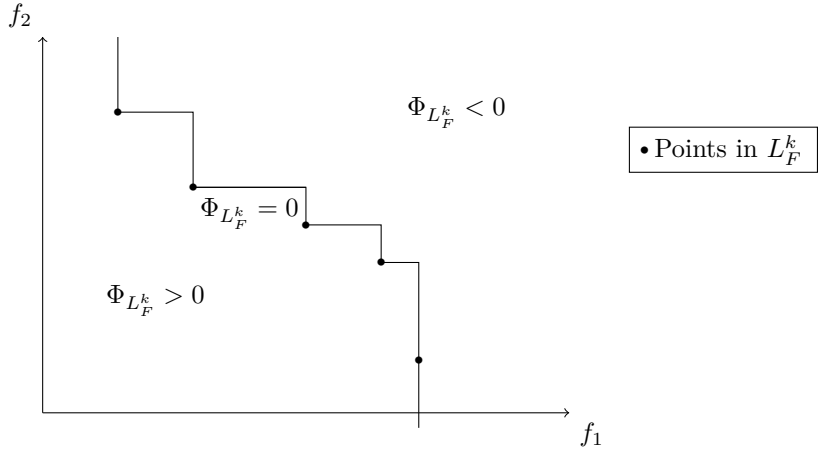
 
 	The current infeasible frame center is then chosen as the element of \(L^k_I\) which maximizes the \(\psi_{L^k_F}\) function, i.e.
 	\[(x^k_I, \Delta^k_I) \in \arg \max_{(x, \Delta) \in L^k_I} \psi_{L^k_F}(x).\]
 	Intuitively, exploring around an infeasible frame center with a large positive value can lead to the generation of a feasible point which significantly improves the current Pareto front approximation. If the selected infeasible frame center possesses a negative value, it is ``close'' to the non-dominated zone relative to the current Pareto front approximation. An exploration around it can still improve the current feasible set.

    \subsubsection{Update of the mesh parameter at the end of an iteration}\label{subsect:update_end_iteration}
    
    At the end of the search and the poll at iteration \(k\), DMulti-MADS-PB has evaluated a finite number of candidates on the mesh \(M^k\). The cache \(V^{k+1}\) is then the union of the cache \(V^k\) at the beginning of iteration \(k\) and all the candidates evaluated during iteration \(k\). As for MADS-PB~\cite{AuDe09a}, 
    the values of \(f\) and \(h\) stored in \(V^{k+1}\) for DMulti-MADS-PB determine the way the threshold value \(h^{k+1}_{\max}\) (see~\eqref{eq:hmax_formula}) and the mesh and frame size parameters of the elements of iterate lists of feasible and infeasible incumbents \(L^{k+1}_F\) and \(L^{k+1}_I\) are updated.
    
    Similarly to MADS-PB~\cite{AuDe09a}, this work uses the concept of \textit{dominating}, \textit{improving} and \textit{unsuccessful} iteration. A dominating iteration occurs when DMulti-MADS-PB generates a trial point which dominates a current frame incumbent.  An improving iteration is not dominating but improves the feasibility of the infeasible frame center. Otherwise, the iteration is unsuccessful. More precisely,
    \begin{itemize}
    	\item Iteration \(k\) is said to be dominating whenever a trial point \(x^t \in V^{k+1}\) dominates one frame incumbent, i.e.
    	\[h(x^t) = 0 \text{ and } x^t \prec_f x^k_F \text{ or } h(x^t) > 0 \text{ and } x^t \prec_h x^k_I\]
    	is found.
    	\item Iteration \(k\) is said to be improving if it is not dominating, but generates a trial point \(x^t \in V^{k+1}\) which satisfies
    	\[0 < h(x^t) < h(x^k_I) \text{ and there exists } i_0 \in \{1,2,\ldots,m\} \text{ such that } f_{i_0}(x^k_I) < f_{i_0}(x^t).\]
    	\item Iterations which are neither dominating nor improving are labelled as unsuccessful. It happens when every trial point \(x^t \in V^{k+1}\) is such that
    	\[h(x^t) = 0 \text{ and } x^t \nprec_f x^k_F, \text{ or } h(x^t) = h(x^k_I) \text{ and } x^t \nprec_h x^k_I \text{ or } h(x^t) > h(x^k_I).\] 	
    \end{itemize}
    These three cases are described in Figure~\ref{fig:iteration_cases_DMultiMADS_PB}.
    
    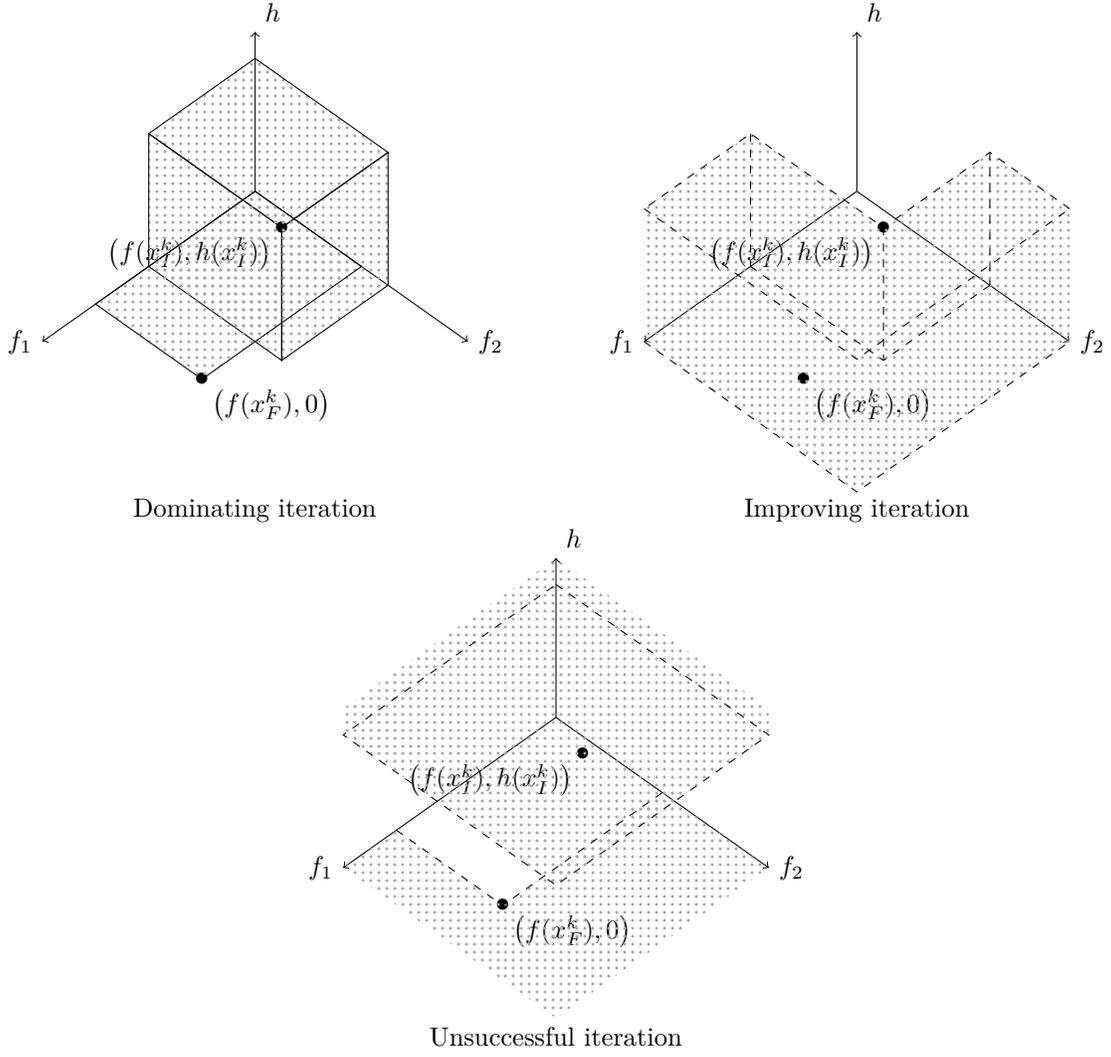
\begin{figure}[!th]
    	\centering
    	\tdplotsetmaincoords{45}{135}
    	\begin{tikzpicture}
    		\begin{scope}[xshift=-8cm, tdplot_main_coords]
    			\draw [->] (0,0,0) -- (4,0,0) node[left]{\small{\(f_1\)}};
    			\draw [->] (0,0,0) -- (0,4,0) node[right]{\small{\(f_2\)}};
    			\draw [->] (0,0,0) -- (0,0,3) node[above right]{\small{\(h\)}};
    			
    			\draw (2,2.5, 2.5) node{\(\bullet\)} node[below left]{\small{\(\left(f(x^k_I), h(x^k_I)\right)\)}};
    			\filldraw [pattern=dots, pattern color=gray!70] (2,2.5, 0) -- (2,2.5,2.5) -- (2, 0,2.5) -- (2,0,0) -- cycle;
    			\filldraw [pattern=dots, pattern color=gray!70] (2,2.5,2.5) -- (0, 2.5,2.5) -- (0,2.5,0) -- (2,2.5,0) -- cycle;
    			\filldraw [pattern=dots, pattern color=gray!70] (2, 2.5, 2.5) -- (0,2.5,2.5) -- (0,0,2.5) -- (2, 0, 2.5) -- cycle;
    			
    			\draw (3, 2, 0) node{\(\bullet\)} node[below right]{\small{\(\left(f(x^k_F), 0\right)\)}};
    			\filldraw[pattern=dots, pattern color=gray!70] (3,2,0) -- (0,2,0) -- (0,0,0) -- (3,0,0) -- cycle;
    			
    			\draw (0,0,-6) node{\small{Dominating iteration}};
    			
    		\end{scope}
    		
    		\begin{scope}[tdplot_main_coords]
    			\draw [->] (0,0,0) -- (4,0,0) node[left]{\small{\(f_1\)}};
    			\draw [->] (0,0,0) -- (0,4,0) node[right]{\small{\(f_2\)}};
    			\draw [->] (0,0,0) -- (0,0,3) node[above right]{\small{\(h\)}};
    			
    			\draw (2,2.5, 2.5) node{\(\bullet\)} node[below left]{\small{\(\left(f(x^k_I), h(x^k_I)\right)\)}};
    			
    			\draw (3, 2, 0) node{\(\bullet\)} node[below right]{\small{\(\left(f(x^k_F), 0\right)\)}};
    			
    			
    			\fill [pattern=dots, pattern color=gray!70] (0, 2.5, 2.5) -- (4,2.5, 2.5) -- (4, 4, 2.5) -- (0,4, 2.5) -- cycle;
    			\fill [pattern=dots, pattern color=gray!70] (4, 4, 2.5) -- (0,4, 2.5) -- (0,4,0) -- (4,4,0) -- cycle;
    			\fill [pattern=dots, pattern color=gray!70] (4, 2.5, 2.5) -- (4,4, 2.5) -- (4,4,0) -- (4, 2.5, 0) -- cycle;
    			
    			\fill [pattern=dots, pattern color=gray!70] (2, 0, 2.5) -- (2,2.5, 2.5) -- (4,2.5,2.5) -- (4, 0, 2.5) -- cycle;
    			\fill [pattern=dots, pattern color=gray!70] (4,2.5,2.5) -- (4, 0, 2.5) -- (4,0,0) -- (4,2.5, 0) -- cycle;
    			
    			\draw[dashed] (0,4,2.5) -- (0, 2.5, 2.5) -- (2, 2.5, 2.5) -- (2,0,2.5) -- (4,0,2.5);
    			\draw[dashed] (0, 2.5, 0) -- (2, 2.5, 0) -- (2,0,0);
    			\draw[dashed] (4,0,0) -- (4,4,0) -- (0,4,0);
    			\draw[dashed] (4,0,2.5) -- (4,4,2.5) -- (0,4,2.5);
    			
    			\draw[dashed] (0, 2.5, 2.5) -- (0,2.5, 0);
    			\draw[dashed] (2,0,2.5) -- (2,0,0);
    			\draw[dashed] (2,2.5,2.5) -- (2,2.5, 0);
    			
    			\draw (0,0,-6) node{\small{Improving iteration}};
    			
    		\end{scope}
    	
    	\begin{scope}[xshift=-4cm,yshift=-7cm,tdplot_main_coords]
    		\draw [->] (0,0,0) -- (4,0,0) node[left]{\small{\(f_1\)}};
    		\draw [->] (0,0,0) -- (0,4,0) node[right]{\small{\(f_2\)}};
    		\draw [->] (0,0,0) -- (0,0,3) node[above right]{\small{\(h\)}};
    		
    		\draw (2,2.5, 2.5) node{\(\bullet\)} node[below left]{\small{\(\left(f(x^k_I), h(x^k_I)\right)\)}};

			\draw (3, 2, 0) node{\(\bullet\)} node[below right]{\small{\(\left(f(x^k_F), 0\right)\)}};
    		
    		
    		\fill [pattern=dots, pattern color=gray!70] (0,0,3) -- (0,4,3) -- (4,4,3) -- (4,0,3) --cycle;
    		\fill [pattern=dots, pattern color=gray!70] (0,4,3) -- (0,4,2.5) -- (4,4,2.5) -- (4,4,3) --cycle;
    		\fill [pattern=dots, pattern color=gray!70] (4,0,3) -- (4,0,2.5) -- (4,4,2.5) -- (4,4,3) --cycle;
    		
    		\fill [pattern=dots, pattern color=gray!70] (4,0,0) -- (3,0,0) -- (3,2,0) -- (0,2,0) -- (0,4,0) -- (4,4,0) -- cycle;
    		
    		\draw[dashed] (0,0,2.5) -- (0,4,2.5) -- (4,4,2.5) -- (4,0,2.5) --cycle;
    		\draw[dashed] (3,0,0) -- (3,2,0) -- (0,2,0);
    		
    		\draw (0,0,-6) node{\small{Unsuccessful iteration}};
    	
    	\end{scope}
    	\end{tikzpicture}
    	\caption{Iterations cases for DMulti-MADS-PB.}
    	\label{fig:iteration_cases_DMultiMADS_PB}
    \end{figure}

    All points generated during iteration \(k\) are given a frame size parameter \(\Delta \geq \Delta^k\) where \(\Delta^k\) is the frame size parameter associated to \(M^k\). More precisely, for any trial element \((x^t, \Delta)\) generated during iteration \(k\),
    \[(x^t, \Delta) = \begin{cases}
    	(x^t, \tau^{-1} \Delta^k) & \text{ if } h(x^t) = 0 \text{ and there exists at least } x \in F^k \text{ such that } x \prec_f x^t, \text{ or } \\
    	(x^t, \tau^{-1} \Delta^k) & \text{ if } h(x^t) = 0 \text{ and for } i = 1,2, \ldots, m,  f_i(x^t) \leq \min_{x \in F^k} f_i(x) \text{ with at least an }  \\
    	& \text{ index } i_0 \in \{1,2, \ldots, m\} \text{ such that } f_{i_0}(x^t) < \min_{x \in F^k} f_{i_0}(x), \text{ or } \\
    	(x^t, \tau^{-1} \Delta^k)  & \text{ if } h(x^t) > 0 \text{ and there exists at least } x \in I^k \text{ such that } x^t \prec_h x, \text{ or} \\
    	(x^t, \tau^{-1} \Delta^k) & \text{ if } 0 < h(x^t) \leq \max_{x \in I^k} h(x) \text{ and  for } i = 1,2,\ldots, m, f_i(x^t) \leq \min_{x \in I^k} f_i(x)\\
    	& \text{ with at least one index } i_0 \in \{1,2, \ldots, m\} \text{ such that } f_{i_0}(x^t) < \min_{x \in I^k} f_{i_0}(x),\\
    	
    	(x^t, \Delta^k) & \text{ otherwise}; \\
    \end{cases}\]
	where \(\tau \in (0,1) \cap \mathbb{Q}\) is the frame size adjustment parameter chosen by the user. Thus, all candidates which dominate one of the points in \(L^k_F\) or \(L^k_I\) or improve the extent of the objectives values covered by at least one of the iterate list have their associated frame size parameter increased. When \(L^k_F\) is empty and no feasible point has been generated at iteration \(k\), these candidates are likely to be potential frame center candidates at iteration \(k+1\). If \(L^k_F\) is not empty, the update of the frame size parameter associated to a new feasible generated point is similar to the one proposed in the original DMulti-MADS-EB algorithm~\cite{BiLedSa2020}. If the iteration is labelled as unsuccessful, no generated point at the end of iteration \(k\) dominates at least one of the frame center incumbents. In this case, DMulti-MADS-PB replaces \((x^k_{center}, \Delta^k_{center})\) by \((x^k_{center}, \tau \Delta^k)\) with \(\tau \in (0, 1) \cap \mathbb{Q}\) and \(x^k_{center} \in \{\{x^k_F\}, \{x^k_I\}, \{x^k_F, x^k_I\}\}\) relatively to the emptiness of the iterate lists \(L^k_F\) or \(L^k_I\). If the iteration is improving, the frame size parameters associated to the existing frame center incumbents keep the same value as in iteration \(k\).
	
	Figure~\ref{fig:zone of increasing mesh} illustrates the frame update rules for a biobjective minimization problem in the ``augmented'' objective space (a triobjective space with two objectives \(f_1\), \(f_2\) and the constraint violation function \(h\)). All candidates whose image is outside combined gray areas are affected a frame size parameter \(\Delta := \Delta^k\).
	
		\begin{figure}[!th]
		\centering
		\tdplotsetmaincoords{45}{135}
		\begin{tikzpicture}
			
			\begin{scope}[xshift=-5.5cm, yshift=3.2cm, tdplot_main_coords]
				\draw [->] (0,0,0) -- (5,0,0) node[left]{\small{\(f_1\)}};
				\draw [->] (0,0,0) -- (0,5,0) node[right]{\small{\(f_2\)}};
				\draw [->] (0,0,0) -- (0,0,4.5) node[above right]{\small{\(h\)}};
				
				\draw (0,0, 4) node{\(-\)} node[right]{\small{\(h^k_{\max}\)}};
				
				\draw (1, 4, 0) node{\tiny{\(\bullet\)}};
				\draw (2, 3, 0) node{\tiny{\(\bullet\)}};
				\draw (2.5, 1, 0) node{\tiny{\(\bullet\)}};
				\draw (3.5, 0.5, 0) node{\tiny{\(\bullet\)}};
				
				\draw (0.7, 4.5, 3) node{\tiny{\(\blacklozenge\)}};
				\draw (1.5, 3.1, 3.5) node{\tiny{\(\blacklozenge\)}};
				\draw (3, 0.75, 2) node{\tiny{\(\blacklozenge\)}};
				
				\filldraw [dotted,pattern=dots, pattern color=gray!50] (5,0,0) -- (5,0,3.5) -- (5,0.75,3.5) -- (5,0.75,0) --cycle;
				\filldraw [dotted,pattern=dots, pattern color=gray!50]  (5,0.75,3.5) -- (3,0.75,3.5) -- (3,0.75, 0) -- (5,0.75, 0) --cycle;
				\filldraw [dotted,pattern=dots, pattern color=gray!50] (5,0,3.5) -- (3,0,3.5) -- (3,0.75, 3.5) -- (5,0.75,3.5) --cycle;
				
				\filldraw [dotted, pattern=dots, pattern color=gray!50] (3, 0.75, 0) -- (1.5, 0.75, 0) -- (1.5, 0.75, 2) -- (3, 0.75, 2) -- cycle;
				\filldraw [dotted, pattern=dots, pattern color=gray!50] (1.5, 0.75, 2) -- (3, 0.75, 2) -- (3, 0, 2) -- (1.5, 0, 2) -- cycle;
				
				\filldraw [dotted,pattern=dots, pattern color=gray!50] (1.5, 3.1, 3.5) -- (1.5, 3.1, 0) -- (1.5, 0, 0) -- (1.5, 0, 3.5) --cycle;
				\filldraw [dotted,pattern=dots, pattern color=gray!50] (1.5, 3.1, 3.5) -- (0, 3.1, 3.5) -- (0,0,3.5) -- (1.5, 0, 3.5) --cycle;
				\filldraw [dotted,pattern=dots, pattern color=gray!50] (1.5, 3.1, 3.5) -- (0, 3.1, 3.5) -- (0,3.1, 0) -- (1.5, 3.1, 0) --cycle;
				
				\filldraw [dotted,pattern=dots, pattern color=gray!50] (0.7, 4.5, 3) -- (0.7, 3.1, 3) -- (0.7, 3.1, 0) -- (0.7, 4.5, 0) --cycle;
				\filldraw [dotted,pattern=dots, pattern color=gray!50] (0.7, 4.5, 3) -- (0, 4.5, 3) -- (0, 3.1, 3) -- (0.7, 3.1, 3) --cycle;
				
				\filldraw [dotted,pattern=dots, pattern color=gray!50] (0.7, 4.5, 3.5) -- (0.7, 5, 3.5) -- (0.7, 5, 0) -- (0.7, 4.5, 0) --cycle;
				\filldraw [dotted,pattern=dots, pattern color=gray!50] (0.7, 4.5, 3.5) -- (0, 4.5, 3.5) -- (0, 5, 3.5) -- (0.7, 5, 3.5) --cycle;
				\filldraw [dotted,pattern=dots, pattern color=gray!50] (0.7, 5, 3.5) -- (0, 5, 3.5) -- (0, 5, 0) -- (0.7, 5,0) --cycle;
				
				\fill [pattern=north west lines, pattern color=gray!50] (1,5, 0) -- (1,4,0) -- (1, 3, 0) -- (2,3,0) -- (2, 1,0) -- (2.5, 1, 0) -- (2.5, 0.5, 0) -- (3.5, 0.5, 0) -- (5, 0.5, 0) -- (5, 0, 0) -- (0,0,0) -- (0,5,0) --cycle;
								
			\end{scope}
			
			
			\draw [->] (0,0) -- (5.5,0) node[below right]{\small{\(f_1\)}};
			\draw [->] (0,0) -- (0,6) node[above left]{\small{\(f_2\)}};
			
			\draw (1, 4) node{\tiny{\(\bullet\)}};
			\draw (2, 3) node{\tiny{\(\bullet\)}};
			\draw (2.5, 1) node{\tiny{\(\bullet\)}};
			\draw (3.5, 0.5) node{\tiny{\(\bullet\)}};
			
			\draw (0.7, 4.5) node{\tiny{\(\blacklozenge\)}};
			\draw (1.5, 3.1) node{\tiny{\(\blacklozenge\)}};
			\draw (3, 0.75) node{\tiny{\(\blacklozenge\)}};
			
			\fill [pattern=dots, pattern color=gray!50] (0.7, 6) -- (0.7, 4.5) -- (1.5, 4.5) -- (1.5, 3.1) -- (3, 3.1) -- (3, 0.75) -- (5.5, 0.75) -- (5.5, 0) -- (0,0) -- (0,6) --cycle;
			
			\fill [pattern=north west lines, pattern color=gray!50] (1, 6) -- (1,4) -- (2,4) -- (2,3) -- (2.5,3) -- (2.5,1) -- (3.5,1) -- (3.5, 0.5) -- (5.5, 0.5) -- (5.5, 0) -- (0,0) -- (0,6) --cycle;
			
			\draw (-10, -1.5) node{\tiny{\(\bullet\)}} node[right]{\small{Set of feasible non dominated points \(F^k\)}};
			\draw (-10, -2) node{\tiny{\(\blacklozenge\)}} node[right]{\small{Set of infeasible incumbent solutions \(I^k\)}};
			\fill [pattern=north west lines, pattern color=gray!50] (-2.8, -1.7) rectangle (-2, -1.15);
			\draw (-2, -1.5) node[right]{\small{Zone of increasing mesh for feasible candidates}};
			\fill [pattern=dots, pattern color=gray!50] (-2.8, -2.3) rectangle (-2, -1.75);
			\draw (-2, -2) node[right]{\small{Zone of increasing mesh for infeasible candidates}};
			
		\end{tikzpicture}
		\caption{An example of an increasing zone for frame size parameters at iteration \(k\) for a biobjective minimization problem. On the left, a 3D view of the ``augmented'' objective space (a triobjective space with the two objectives \(f_1\), \(f_2\) and the constraint violation function \(h\)); on the right, projection on the biobjective space.}
		\label{fig:zone of increasing mesh}
	\end{figure}
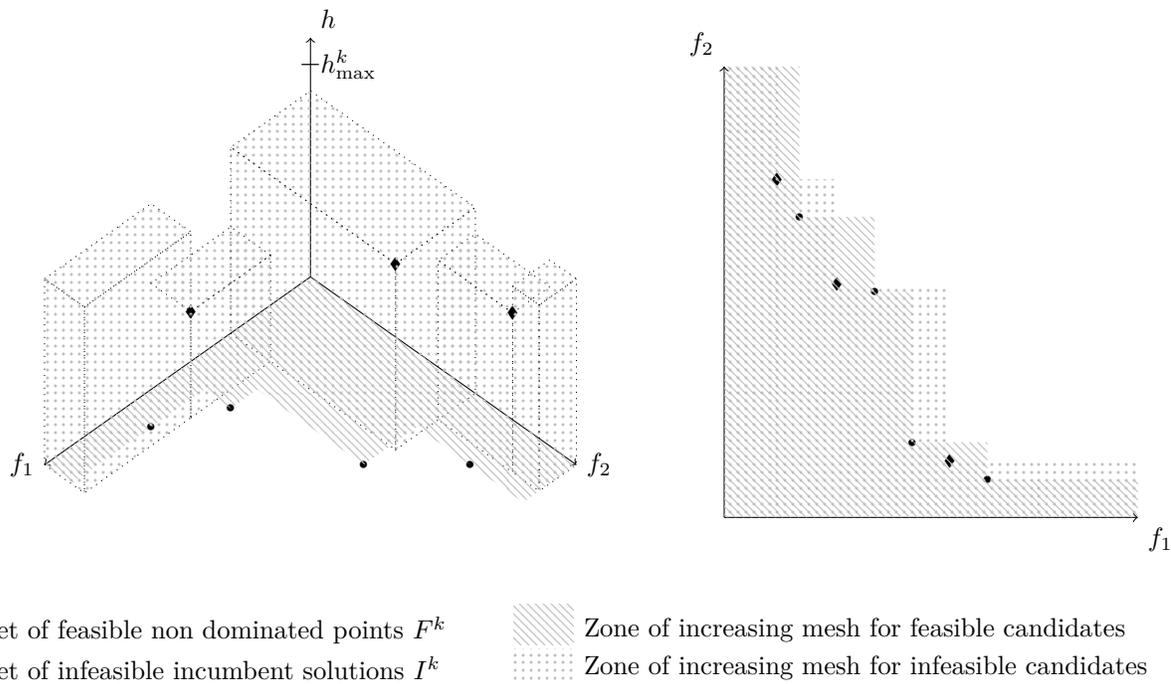

	The threshold barrier is then updated according to the following rules:
	\begin{equation}
	h^{k+1}_{\max} := \begin{cases}
				\displaystyle\max_{x^t \in V^{k+1}} \left\{h(x^t) : h(x^t) < h(x^k_I)\right\} & \text{ if iteration } k \text{ is improving}, \\
				h(x^k_I) & \text{ if } h(x^k_I) = \displaystyle\max_{x \in I^k} h(x),  \\
				\displaystyle\max_{x^t \in V^{k+1}} \left\{ h(x^t) : h(x^k_I) \leq h(x^t) < \max_{x \in I^k} h(x) \right\} & \text{ otherwise.}
				\end{cases}
    \label{eq:hmax_formula}
    \end{equation}

	The threshold update rule guarantees in the case where an iteration is considered as not improving that the set \(I^k\) will change if the infeasible frame incumbent does not possess the maximum violation function value \(h\) among the elements of \(I^k\) at iteration \(k\). Another consequence (similar to MADS-PB~\cite{AuDe09a}) is that \(h^k_{\max}\) is nonincreasing with iteration \(k\) and that if \(I^k \neq \emptyset\), \(I^q \neq \emptyset\) for all iteration indexes \(q \geq k\).
	
	Note that even if an iteration is marked as unsuccessful, the algorithm can still generate new feasible non-dominated points or infeasible non-dominated points below the value \(\max_{x \in I^k} h(x)\), which may be used as frame incumbents in some next iteration.

	
	\begin{remark}
		It is also possible to set the update rules of the threshold \(h^{k+1}_{\max}\) according to the \(h(x^k_I)\) barrier value. Nonetheless, in some preliminary experiments, it has been observed that this approach prevents the algorithm to explore some parts of the objective space, potentially interesting to greatly improve the current feasible solution set. 
	\end{remark}

	Finally, the iterate lists \(L^k_F\) and \(L^k_I\) are filtered to add new non-dominated points generated during iteration \(k\) and remove potential resulting dominated elements.
	
	Algorithm~\ref{alg:summary_DMulti-MADS-PB_algorithm} summarizes the different steps of the DMulti-MADS-PB algorithm.

	\begin{figure}[!th]
		\begin{algorithm}[H]\small
			\caption{The DMulti-MADS-PB algorithm for constrained optimization}
			\begin{algorithmic}
				\STATE \textbf{Initialisation} : Given a finite set of points \(V^0 \subset \mathcal{X}\), choose \(\Delta^0 > 0\), \(D = G Z\) a positive spanning set matrix, \(\tau \in (0, 1) \cap \mathbb{Q}\) the frame size adjustment parameter, and \(w^+ \in \mathbb{N}\) a fixed integer parameter. Define the frame trigger parameter \(\rho > 0\) (optional). Initialize the lists \(L^0_F = \left\{ (x^l_F, \Delta^0), \ l = 1, 2, \ldots, |L^0_F|\right\}\) and \(L^0_I = \left\{ (x^l_I, \Delta^0), \ l = 1, 2, \ldots, |L^0_I|\right\}\) for some \((x^l_F, x^l_I) \in V^0\).
				\FOR{\(k = 0,1,2, \ldots\)}
				\STATE \textbf{Selection of the current infeasible frame centers}. Select feasible and/or infeasible elements of respective iterate lists \(L^k_F\) and \(L^k_I\) as described in~\cite{BiLedSa2020} and Algorithm~\ref{alg:select_point_strategy_infeasible_nofeasible}. Define the current frame size parameter \(\Delta^k\) according to the associated frame size parameters of the feasible incumbent element \((x^k_F, \Delta^k_F)\) and/or infeasible current incumbent element \((x^k_I, \Delta^k_I)\). Set \(\delta^k = \min \left\{\Delta^k, \left(\Delta^k\right)^2\right\}\). Initialize \(L^{add} := \emptyset\).
				\STATE \textbf{Search} (optional) : Evaluate \(f\) and \(h\) at a finite set of points \(S^k \subset \mathcal{X}\) on the mesh \(M^k = \bigcup_{x \in V^k} \{x + \delta^k D z : z \in \mathbb{N}^{n_D}\}\). Set \(L^{add} := \{(x, \Delta^k): x \in S^k\}\).
				
				If an improving or dominating success criterion is satisfied, the search may terminate. In this case, skip the poll and go to the parameter update step.
				\STATE \textbf{Poll} : Select a positive spanning set \(\mathbb{D}^k_{\Delta}\). Evaluate \(f\) and \(h\) on the poll set \(P^k \subset M^k\) as defined in Subsection~\ref{subsect:iteration_dmulti-mads-pb}. Set \(L^{add} := L^{add} \cup \left\{(x, \Delta^k) : x \in P^k \right\}\).
				If an improving or dominating criterion is satisfied, the poll may terminate opportunistically.
				\STATE \textbf{Parameter update} : Define \(V^{k+1}\) as the union of \(V^k\) and all new candidates evaluated in \(\mathcal{X}\) during the search and the poll. Classify the iteration as dominating, improving or unsuccessful. Update \(h^{k+1}_{\max}\) according to Section~\ref{subsect:update_end_iteration}. Remove points above the threshold from \(L^k_I\). Update iterate lists \(L^{k+1}_F\) and/or \(L^{k+1}_I\) by adding new non-dominated points from \(L^{add}\) with their updated associated frame center \(\Delta \in \{\Delta^k, \tau^{-1} \Delta^k\}\), as explained in Section~\ref{subsect:update_end_iteration}. Remove new dominated points from \(L^k_F\) and/or \(L^k_I\).
				
				If the iteration is unsuccessful, replace (if they exist) the frame center elements \((x^k_F, \Delta^k_F)\) and \((x^k_I, \Delta^k_I)\) respectively by \((x^k_F, \Delta^{k+1})\), \((x^k_I, \Delta^{k+1})\) with \(\Delta^{k+1} := \tau \Delta^k\).
				\ENDFOR 
			\end{algorithmic}
			\label{alg:summary_DMulti-MADS-PB_algorithm}
		\end{algorithm}
		\caption{A summary of the DMulti-MADS-PB algorithm, inspired by~\cite{BiLedSa2020}.}
	\end{figure}

	\begin{remark}
		When \(m = 1\), the classification of the different type of iterations used in the DMulti-MADS-PB context is equivalent to the one used for the MADS-PB algorithm~\cite{AuDe09a}. There also exists many configurations of iteration classifications criteria such that the generalization of the MADS-PB algorithm for multiobjective optimization and the convergence properties still hold. For example, one can declare an iteration as dominating when a trial point changes the set \(I^k\). Practically, not all of them have the same performance. The definitions used below correspond to the most efficient variant observed on some preliminary experiments.
	\end{remark}

	\subsubsection{A frame center selection rule for the DMulti-MADS-PB algorithm}\label{subsect:frame_center_selection}
		
	The constrained single-objective MADS-PB algorithm uses a classification of its two frame centers to practically improve the performance of the poll step. More precisely, the two frame centers are ordered, based on their objective values into \textit{primary} and \textit{secondary} poll centers. MADS-PB concentrates more efforts (based on the number of poll directions) on the primary poll center than the secondary poll center~\cite{AuDe09a}. Inspired by this strategy, this subsection proposes an extension of the so-called \textit{frame center selection rule} to constrained multiobjective optimization.
	
	As in the single-objective case, DMulti-MADS-PB executes the poll around at least two frame centers. When \(L^k_F = \emptyset\) or \(L^k_I = \emptyset\), there is only one frame center, designed as the \textit{primary frame center}. A complete set of poll points can be evaluated based on a positive spanning set \(\mathbb{D}^k_{\Delta}\) composed of at least of \(n+1\) directions (more details for the construction of \(\mathbb{D}^k_{\Delta}\) can be found in~\cite{AbAuDeLe09, AuIaLeDTr2014}).
	
	When \(L^k_F\) and \(L^k_I\) are both non-empty, polling is done around a feasible and an infeasible frame centers. DMulti-MADS-PB orders these two frame centers into a \textit{primary} frame center and a \textit{secondary} frame center. This ordering is based on an user-supplied parameter \(\rho > 0\), called the \textit{frame trigger parameter}.
	
	Recall that if \(L^k_F\) and \(L^k_I\) are nonempty, the selection of the infeasible frame center is done based on the \(\psi_{L^k_F}:\mathcal{X} \rightarrow \mathbb{R}\) function parametrized by \(L^k_F\), defined in Section~\ref{subsect:iteration_dmulti-mads-pb}. The following frame center selection rule is then proposed.
	
	\begin{definition}[frame center selection rule]
		Let \(\rho > 0\) provided by the user and suppose that \(L^k_F \neq \emptyset\) and \(L^k_I \neq \emptyset\). Let \((x^k_F, \Delta^k_F) \in L^k_F\) the feasible current incumbent and \((x^k_I, \Delta^k_I) \in L^k_I\) the infeasible current incumbent.  If \(\psi_{L^k_F}(x^k_I) - \rho \ \xi(L^k_F) > 0\), where \(\xi(L^k_F)\) is given by
		\[\xi(L^k_F) = \sum_{i = 1}^{m} \mu\left(\max_{(x, \Delta) \in L^k_F} f_i(x), \min_{(x, \Delta) \in L^k_F} f_i(x)\right)\]
		with \(\mu: \mathbb{R} \times \mathbb{R} \rightarrow \mathbb{R}^+\) defined as
		\[\mu(a, b) = \begin{cases}
			|a - b| & \text{ if } a \neq b, \\
			|a| & \text{ otherwise}; 
		\end{cases}\]
		then the primary poll center is chosen as \(x^k_I\) and the secondary poll center is chosen as \(x^k_F\), otherwise the primary poll center is chosen as \(x^k_F\) and the secondary poll center is chosen as \(x^k_I\).
	\end{definition}
	As for the single-objective MADS-PB algorithm, DMulti-MADS-PB puts more effort on the primary frame center than on the secondary frame center. The implementation of the poll strategy in this work follows the one developed in~\cite{AuIaLeDTr2014}: \(n+1\) directions are used for the primary frame center and \(2\) directions for the secondary frame center by taking the negative of the first one.
	
	If there exists at least one element \((x, \Delta) \in L^k_F\) such that \(f_i(x) \leq f_i(x^k_I)\) for \(i = 1,2, \ldots,m\), then \(x^k_F\) will be chosen as the primary poll center. Figure~\ref{fig:primary_poll_zone_Ik} illustrates the zone in the biobjective space where \(I^k\) elements must be to be considered as potential primary poll centers. One could hope that putting more effort on the infeasible frame center in this case should enable it to reach a better part of the feasible decision region~\cite{AuDe09a}.
	
	\begin{figure}[!th]
		\centering
		\begin{tikzpicture}
			\draw [->] (0,0) -- (8,0) node[below right]{\small{\(f_1\)}};
			\draw [->] (0,0) -- (0,6) node[above left]{\small{\(f_2\)}};
			
			\draw (1.5, 5) node{\tiny{\(\bullet\)}};
			\draw (3, 3.5) node{\tiny{\(\bullet\)}};
			\draw (4.5, 2.5) node{\tiny{\(\bullet\)}};
			\draw (7, 1.5) node{\tiny{\(\bullet\)}};
			
			\draw[dashed, <-] (1.5, 5) -- (7, 5) node[pos=0.7, above]{\small{\(\xi(L^k_F)\)}};
			\draw[dashed, ->] (7,5) -- (7,1.5);
			\draw[thick, <->] (7, 3.2) -- (7,2.4) node[midway, right]{\small{\(\rho \ \xi(L^k_F)\)}};

			\fill[pattern=dots, pattern color=gray!60] (9,4.7) rectangle (10.5, 5.3);
			\draw (10.5, 5) node[right]{\small{Primary poll zone for \(I^k\)}};
			\draw (10.5, 4) node{\tiny{\(\bullet\)}} node[right]{\small{Set of feasible non dominated}};
			\draw (10.5, 3.5) node[right]{\small{points \(F^k\)}};
			
			\fill[pattern=dots, pattern color=gray!60] (0,6) -- (0.7, 6) -- (0.7, 5) -- (1.5, 4.2) -- (2.2, 4.2) -- (2.2, 3.5) -- (3, 2.7) -- (3.7, 2.7) -- (3.7, 2.5) -- (4.5, 1.7) -- (6.2, 1.7) -- (6.2, 1.5) -- (7, 0.7) -- (8,0.7) -- (8, 0) -- (0,0) -- cycle;
			
			\draw[thick, <->] (0.7,5) -- (1.5,5);
			\draw[thick, <->] (1.5,5) -- (1.5,4.2);
			
			\draw[thick, <->] (2.2,3.5) -- (3,3.5);
			\draw[thick, <->] (3,2.7) -- (3,3.5);
			
			\draw[thick, <->] (3.7,2.5) -- (4.5,2.5);
			\draw[thick, <->] (4.5,1.7) -- (4.5,2.5);
			
			\draw[thick, <->] (6.2,1.5) -- (7,1.5);
			\draw[thick, <->] (7,0.7) -- (7,1.5);
			
		\end{tikzpicture}
		\caption{Representation of the selection of \(I^k\) frame incumbent as primary poll in the objective space for a biobjective minimization problem.}
		\label{fig:primary_poll_zone_Ik}
	\end{figure}
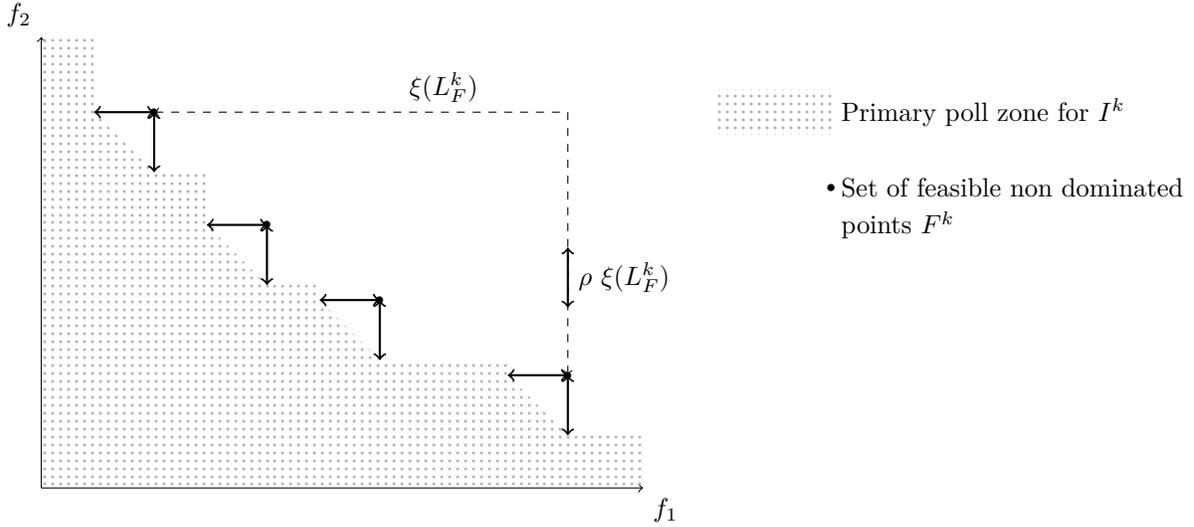

	\begin{remark}
		If \(m = 1\), the frame center selection rule is equivalent to \(f(x^k_I) < f(x^k_F) - \rho | f(x^k_F) |\), used in~\cite{AuCoLedPey2016}. In this work, this rule was privileged to the original one in~\cite{AuCoLedPey2016} \(f(x^k_F) - f(x^k_I) > \rho\), as it takes into account the scale of the objective function values. A corresponding frame center selection rule extension to multiobjective optimisation would have been \(\psi_{L^k_F}(x^k_I) - m \rho > 0\).
	\end{remark}

	\section{Convergence analysis}\label{sect:Convergence_analysis}
	
	This section is devoted to the convergence analysis of the DMulti-MADS-TEB and DMulti-MADS-PB algorithms, inspired by~\cite{AuDe09a, BiLedSa2020}. This work makes use of the following assumptions, taken from~\cite{AuDe09a}:
	\begin{assumption}~\label{assumption:starting_points}
		There exists some point \(x^0\) in the user-provided set \(V^0\) such that \(x^0 \in \mathcal{X}\) and \(f(x^0)\) and \(h(x^0)\) are both finite.
	\end{assumption}

	\begin{assumption}~\label{assumption:boundness_properties}
		All iterates considered by the algorithm lie in a bounded set.
	\end{assumption}

	If Assumption~\ref{assumption:starting_points} is not satisfied, DMulti-MADS cannot start. Assumption~\ref{assumption:boundness_properties}  is ensured if one imposes the existence of a bounded set in \(\mathbb{R}^n\) containing \(V^k\) for all \(k \in \mathbb{N}\). As \(V^k\) for \(k \in \mathbb{N}\) is always composed of points satisfying the unrelaxable constraints, it is sufficient to guarantee that the set of unrelaxable constraints is itself a bounded set. For example, many engineering problems possess bound variables constraints, which cannot be violated.
	
	As in single-objective MADS-PB algorithm~\cite{AuDe09a}, combining assumptions~\ref{assumption:starting_points} and~\ref{assumption:boundness_properties} and the structure of the mesh \(M^k\) enables to show that \(\displaystyle\lim_{k \rightarrow + \infty} \inf \Delta^k = \displaystyle\lim_{k \rightarrow + \infty} \inf \delta^k = 0\) (see for example~\cite[Theorem~ A.1]{CuMaVaVi2010}). The classical convergence analysis of direct search methods focuses on subsequences of generated frame centers for which corresponding mesh size and frame size parameters converge to zero. The following notations and definitions are adapted from~\cite{AuDe09a}.

	Let \(\mathbb{U} \subset \mathbb{N}\) the set of unsuccessful iterations indexes. The poll generates one or several trial points around at least one of the two feasible and infeasible incumbents. If \(k \in \mathbb{U}\) and the poll is executed around the feasible current frame center \(x^k_F\), this last one is designed as a \textit{feasible minimal frame center}. Otherwise, if \(k \in \mathbb{U}\) and the poll is executed around the infeasible current frame center \(x^k_I\), this last one is designed as an \textit{infeasible minimal frame center}. From the rest of this work, these subsequences of frame centers are investigated separately. Note that for the DMulti-MADS-TEB variant, studying a subsequence of infeasible minimal frame centers means that the algorithm does not manage to find a feasible point.
	
	\begin{definition}
		A subsequence \(\{x^k\}_{k \in K}\) of DMulti-MADS frame centers, for some infinite subset of indexes \(K \subseteq \mathbb{U}\) is said to be a \textit{refining subsequence} if \(\{\Delta^k\}_{k \in K}\) converges to \(0\). The limit point \(\hat{x}\) of a refining subsequence is called a \textit{refining point}. 
	\end{definition}

	\begin{definition}
		Given a corresponding refining subsequence \(\{x^k\}_{k \in K}\) and its refining point \(\hat{x}\), a direction \(d\) is said to be a \textit{refining direction} if and only if there exists an infinite subset of indexes \(K' \subseteq K\) such that \(d^k \in \mathbb{D}^k_{\Delta}\) with \(x^k + \delta^k d^k \in \mathcal{X}\) and \(d = \lim_{k \in K'} \frac{d^k}{\|d^k\|}\).
	\end{definition}

	The convergence analysis also requires some mathematical tools from nonsmooth analysis. The following definitions are taken from~\cite{AuDe09a}.
	
	\begin{definition}
		A vector \(d \in \mathbb{R}^n\) is said to be a Clarke tangent vector to the set \(\Omega \subseteq \mathbb{R}^n\) at the point \(x\) in the closure of \(\Omega\) if for every sequence \(\{y^k\}\) of elements of \(\Omega\) that converge to \(x\) and for every sequence of positive real numbers \(\{t^k\}\) converging to zero, there exists a sequence of vectors \(\{w^k\}\) converging to \(d\) such that \(y^k + t^k w^k \in \Omega\).
	\end{definition}

	The set of all Clarke tangent cones to \(\Omega\) at \(x\) is the Clarke tangent cone to \(\Omega\) at \(x\) denoted as \(T^{Cl}_{\Omega}(x)\). The DMulti-MADS analysis in a general constrained optimization context makes use of the hypertangent cone~\cite{Rock1974}, which is the interior of the Clarke tangent cone, defined as:
	\[T^H_{\Omega}(x) = \left\{d \in \mathbb{R}^n : \exists \ \epsilon > 0  \text{ such that } y + t w \in \Omega, \text{ for all } y \in \Omega \cap B_{\epsilon}(x), w \in B_{\epsilon}(d), \text{ and } 0 < t < \epsilon \right\}\]
	where \(B_{\epsilon}(x)\) is the open ball of radius \(\epsilon > 0\) centered at \(x\).
	
	The DMulti-MADS analysis also requires that the objective function \(f\) is locally Lipschitz continuous in \(\mathcal{X}\), i.e. each of its components \(f_i\) for \(i = 1,2, \ldots, m\) is locally Lipschitz continuous in \(\mathcal{X}\). If this condition is satisfied, the Clarke-Jahn generalized derivatives~\cite{Clar83a} of \(f_i\) at \(x \in \mathcal{X}\) in the direction \(d \in \mathbb{R}^n\) exist and are defined by
	\[f^o_i(x; d) = \underset{\begin{array}{c}y \rightarrow x, \ y \in \mathcal{X}\\ t \searrow 0, \ y + td \in \mathcal{X} \end{array}}{\lim \sup} \dfrac{f(y + t d) - f(y)}{t}, \text{ for } i = 1,2, \ldots, m.\]

	This work can then introduce the main stationary conditions.
	
	\begin{definition}
		Let \(f\) be Lipschitz continuous near a point \(\hat{x} \in \Omega\). \(\hat{x}\) is a Pareto-Clarke critical point of \(f\) in \(\Omega\) if for all directions \(d \in T^{Cl}_{\Omega}(\hat{x})\), there exists \(i(d) \in \{1,2, \ldots, m\}\) such that \(f^o_i(d)(\hat{x};d) \geq 0\).
	\end{definition}

	With the additional assumption that \(f\) is equally strictly differentiable at \(\hat{x}\) (i.e. the corresponding Clarke generalized is a singleton containing only the gradient of one objective component at \(\hat{x}\)), the previous stationary result can be reformulated.
	
	\begin{definition}
		Let \(f\) be strictly differentiable at a point \(\hat{x} \in \Omega\). \(\hat{x}\) is a Pareto-Clarke-KKT critical point of \(f\) in \(\Omega\) if for all directions \(d \in T^{Cl}_{\Omega}(\hat{x})\), there exists \(i(d) \in \{1,2, \ldots, m\}\) such that \(\nabla f_{i(d)}(\hat{x})^\top d \geq 0\).
	\end{definition}

	As in the single-objective case~\cite{AuDe09a}, this work divides the convergence analysis into two cases: the study of subsequences of feasible minimal frame centers and the study of subsequences of infeasible minimal frame centers. For each case, the following methodology is used:
	\begin{enumerate}
		\item Prove that a subsequence of mesh size parameters and frame size parameters converges to zero.
		\item Determine a particular subsequence of iterate points associated to the previous subsequence of parameters, i.e. a so-called \textit{refined subsequence}. 
		\item This subsequence of iterate points converges to a refined point. Prove that this point satisfies some stationary properties.
	\end{enumerate}

	\subsection[Feasible case: results for f]{Feasible case: results for \(f\)}
	
	As in~\cite{BiLedSa2020}, one wants to show that starting from a set of feasible points, DMulti-MADS produces at the limit locally stationary points for the constrained multiobjective optimization problem. To do that, this work proves the existence of finer refining subsequences, as it is done in~\cite{BiLedSa2020}.  The following analysis is a summary of the convergence analysis developed in~\cite{BiLedSa2020} and covers the two variants DMulti-MADS-TEB and DMulti-MADS-PB.
	
	\begin{theorem}~\label{thm:max_delta_inf_decrease}
		Let Assumptions~\ref{assumption:starting_points} and~\ref{assumption:boundness_properties} hold and suppose DMulti-MADS generates a sequence of feasible iterates lists \(\{L^k_F\}\) with \(L^k_F = \{(x^j, \Delta^j), j = 1, 2 \ldots, |L^k_F|\}\). Then
		\[\lim_{k \rightarrow + \infty} \inf \delta^k_{\max} = \lim_{k \rightarrow + \infty} \inf \Delta^k_{\max} = 0.\]
	\end{theorem}
	\begin{proof}
		Combining assumptions~\ref{assumption:starting_points}, ~\ref{assumption:boundness_properties}, and the selection criterion of the feasible frame center with the structure of the mesh has been shown to be enough to ensure \(\displaystyle\lim_{k \rightarrow + \infty} \inf \delta^k_{\max} = \displaystyle\lim_{k \rightarrow + \infty} \inf \Delta^k_{\max} = 0\) (see~\cite[Theorem 5.1]{BiLedSa2020} for more details). 
	\end{proof}

	This work wants to prove the convergence of specific elements of the feasible iterate list generated by DMulti-MADS to stationary points. The concept of a feasible linked list, adapted from~\cite{BiLedSa2020, LiLuRi2016}, is then introduced.
	
	\begin{definition} Suppose DMulti-MADS generates the sequence of feasible iterate lists \(\{L^k_F\}_{k \geq k_0}\) with \(L^k_F = \{(x^l, \Delta^l), l = 1,2, \ldots, |L^k_F|\}\) and \(k_0 \in \mathbb{N}\) the iteration index such that \(k_0 \in \arg\min \left\{ k \in \mathbb{N} : F^k \neq \emptyset \right\}\). A \textit{feasible linked sequence} is defined as a sequence \(\{(x^{l_k}, \Delta^{l_k})\}\) such that there exists an iteration index \(\ell_0 \geq k_0\) such that for any \(k = \ell_0+1, \ell_0+2, \ldots\), the pair \(\left\{(x^{l_k}, \Delta^{l_k})\right\} \in L^k_F\) is generated at iteration \(k-1\) of DMulti-MADS from the pair \((x^{l_{k-1}}, \Delta^{l_{k-1}}) \in L^{k-1}_F\).
	\end{definition}

	For the DMulti-MADS-PB variant algorithm, the following cases can occur:
	\begin{enumerate}
		\item Dominating iteration: either the algorithm generates at least one point which dominates the feasible frame center \(x^{k-1}_F\), or it generates some infeasible points which have triggered the dominating success condition in the infeasible case.
		\begin{itemize}
			\item \(\forall (x^{l_k}, \Delta^{l_k}) \in L^{k}_F \setminus L^{k-1}_F\),
			\[x^{l_k} = x^{k-1} + \delta^{k-1} D z^{k-1} \text{ for some } z^{k-1} \in \mathbb{N}^{n_D} \text{ and } \Delta^{l_k} \in \{\Delta^{k-1}, \tau^{-1} \Delta^{k-1}\}\]
			with \(x^{k-1} \in \{x^{k-1}_F, x^{k-1}_I\}\).
			\item \(\forall (x^{l_k}, \Delta^{l_k}) \in L^{k}_F \cap L^{k-1}_F\),
			\[x^{l_k} = x^{l_{k-1}} \text{ and } \Delta^{l_k} = \Delta^{l_{k-1}}.\]
		\end{itemize}
		\item Improving iteration: the algorithm may generate some new feasible non-dominated points without dominating the feasible frame incumbent.
		\begin{itemize}
			\item \(\forall (x^{l_k}, \Delta^{l_k}) \in L^{k}_F \setminus L^{k-1}_F\),
			\[x^{l_k} = x^{k-1} + \delta^{k-1} D z^{k-1} \text{ for some } z^{k-1} \in \mathbb{N}^{n_D} \text{ and } \Delta^{l_k} \in \{\Delta^{k-1}, \tau^{-1} \Delta^{k-1}\}\]
			with \(x^{k-1} \in \{x^{k-1}_F, x^{k-1}_I\}\).
			\item \(\forall (x^{l_k}, \Delta^{l_k}) \in L^{k}_F \cap L^{k-1}_F\),
			\[x^{l_k} = x^{l_{k-1}} \text{ and } \Delta^{l_k} = \Delta^{l_{k-1}}.\]
		\end{itemize}
		\item Unsuccessful iteration: the algorithm may generate some new feasible non-dominated points without dominating the feasible frame incumbent.
		\begin{itemize}
			\item \(\forall (x^{l_k}, \Delta^{l_k}) \in L^{k}_F \setminus L^{k-1}_F\),
			\[x^{l_k} = x^{k-1} + \delta^{k-1} D z^{k-1} \text{ for some } z^{k-1} \in \mathbb{N}^{n_D} \text{ and } \Delta^{l_k} \in \{\Delta^{k-1}, \tau^{-1} \Delta^{k-1}\}\]
			with \(x^{k-1} \in \{x^{k-1}_F, x^{k-1}_I\}\).
			\item \(\forall (x^{l_k}, \Delta^{l_k}) \in (L^{k}_F \cap L^{k-1}_F) \setminus \{(x^{k-1}_F, \Delta^{k-1}_F)\}\),
			\[x^{l_k} = x^{l_{k-1}} \text{ and } \Delta^{l_k} = \Delta^{l_{k-1}}.\]
			\item \(\forall (x^{l_k}, \Delta^{l_k}) \in \{(x^{k-1}_F, \Delta^{k-1}_F)\}\),
			\[x^{l_k} = x^{k-1}_F \text{ and } \Delta^{l_k} = \tau \Delta^{k-1}.\]
		\end{itemize}
	\end{enumerate}
	Similar relations can be drawn for the DMulti-MADS-TEB variant algorithm : note that for all \(k > k_0\), no point at iteration \(k\) can be generated from an infeasible point at iteration \(k-1\).
	
	One can then prove that feasible linked sequences contain a feasible refining subsequence. The original proof can be found in~\cite{BiLedSa2020}, but for better understanding, it is restated below.
	
	\begin{theorem}
		Let assumptions~\ref{assumption:starting_points} and~\ref{assumption:boundness_properties} hold and suppose DMulti-MADS generates the sequence of feasible iterate lists \(\{L^k_F\}_{k \geq k_0}\) with \(L^k_F = \{(x^l, \Delta^l), l = 1,2,\ldots, |L^k_F|\}\) and \(k_0 \in \mathbb{N}\) the iteration index such that \(k_0 \in \arg \min \left\{k \in \mathbb{N} : F^k \neq \emptyset\right\}\). Then every feasible linked sequence \(\{(x^{l_k}, \Delta^{l_k})\}\) contains a refining subsequence \(\{x^{l_k}\}_{k \in K}\) for some infinite subset of indexes \(K \subset \mathbb{U}\).
	\end{theorem}
	\begin{proof}
		\(\forall k \geq k_0\), \(0 \leq \Delta^{l_k} \leq \Delta^k_{\max}\). By combining Theorem~\ref{thm:max_delta_inf_decrease} and the squeeze theorem, one gets
		\[\lim_{k \rightarrow + \infty} \inf \Delta^{l_k} = \lim_{k \rightarrow + \infty} \inf \Delta^k_{\max} = 0,\]
		which implies by definition the existence of a refining feasible subsequence within \(\{(x^{l_k}, \Delta^{l_k})\}\). 
	\end{proof}

	The analysis which follows is similar to~\cite{BiLedSa2020}.
	
	\begin{theorem}~\label{thm:direction_clarke_stationary}
		Let assumptions~\ref{assumption:starting_points} and~\ref{assumption:boundness_properties} hold and suppose DMulti-MADS generates a feasible refining subsequence~\(\{x^k_F\}_{k \in K}\), with \(x^k_F \in F^k\), converging to a refining point \(\hat{x}_F \in \Omega\). Assume that \(f\) is Lipschitz continuous near \(\hat{x}_F\). If \(d \in T^H_\Omega(\hat{x}_F)\) is a refining direction for \(\hat{x}_F\), then there exists an objective index \(i(d) \in \{1,2,\ldots,m\}\) such that \(f^o_{i(d)}(\hat{x}_F; d) \geq 0\).
	\end{theorem}
	\begin{proof}
		Let \(\{x^k_F\}_{k \in K}\), with \(x^k_F \in F^k\), be a refining subsequence converging to a feasible refining point \(\hat{x}_F \in \Omega\) and \(d = \lim_{k \in K'} \frac{d^k}{\|d^k\|} \in T^H_{\Omega}(\hat{x}^F)\) a refining direction for \(\hat{x}_F\), where \(K' \subseteq K\) is an infinite subsequence of some infinite subset of unsuccessful iteration indexes, with poll directions \(d^k \in \mathbb{D}^k_\Delta\) such that \(x^k_F + \delta^k d \in \Omega\). Denote by \(\nu \geq 0 \) the Lipschitz constant of \(f\) near \(\hat{x}_F\).
		
		Then, for \(i \in \{1,2,\ldots,m\}\), the inequality
		\[\begin{array}{ccl}
		f(\hat{x}_F; d) & = & f_i^o(\hat{x}_F;d) + \lim \displaystyle\sup_{k \in K'} \dfrac{\nu \ \delta^k \|d^k\| \left\| \frac{d^k}{\|d^k\|} - d \right\|}{\delta^k \|d^k\|} \\
		& \geq & f_i^o(\hat{x}_F;d) + \lim \displaystyle\sup_{k \in K'} \dfrac{|f_i\left(x^k_F + \delta^k d^k\right) - f_i\left(x^k_F + \delta^k \|d^k\| d\right)|}{\delta^k \|d^k\|} \\
		& \geq & \lim \displaystyle\sup_{k \in K'} \dfrac{f_i\left(x^k_F + \delta^k \|d^k\|d\right) - f(x^k_F)}{\delta^k \|d^k\|} \\
		& & + \lim \displaystyle\sup_{k \in K'} \dfrac{|f_i\left(x^k_F + \delta^k d^k\right) - f_i\left(x^k_F + \delta^k \|d^k\| d\right)|}{\delta^k \|d^k\|} \\
		& \geq & \lim \displaystyle\sup_{k \in K'}  \dfrac{f(x^k_F + \delta^k d^k) - f_i\left(x^k_F + \delta^k \|d^k\| d\right) + f_i\left(x^k_F + \delta^k \|d^k\| d\right) - f_i(x^k_F)}{\delta^k \|d^k\|} \\
		& = & \lim \displaystyle\sup_{k \in K'} \dfrac{f_i(x^k_F + \delta^k d) - f_i(x^k_F)}{\delta^k \|d^k\|}		
		\end{array}\]
		is satisfied.
		
		\(\{x^k_F\}_{k \in K}\) being a refining subsequence, the infinite subset of indexes \(K' \subseteq K\) corresponds to unsuccessful iterations. Consequently, the point \(x^k_F + \delta^k d \in \Omega\) does not dominate \(x^k_F\). One can then find an infinite subsequence of indexes \(K'' \subset K'\) such that there exists an index \(i(d) \in \{1,2,\ldots, m\}\) satisfying
		\[f_{i(d)}(\hat{x}_F; d) \geq \lim \displaystyle\sup_{k \in K''} \dfrac{f_{i(d)}(x^k_F + \delta^k d) - f_{i(d)}(x^k_F)}{\delta^k \|d^k\|} \geq 0.\]
	\end{proof}

	When the set of refining directions is dense in a non-empty hypertangent cone at \(\Omega\), Pareto Clarke stationarity is ensured, similarly to the analysis conducted in~\cite{BiLedSa2020, CuMaVaVi2010}.
	
	\begin{theorem}~\label{thm:Pareto_Clarke_stationnarity}
		Let assumptions~\ref{assumption:starting_points} and~\ref{assumption:boundness_properties} and suppose DMulti-MADS generates a feasible refining subsequence \(\{x^k_F\}_{k \in K}\), with \(x^k_F \in F^k\), converging to a refining point \(\hat{x}_F \in \Omega\). Assume that f is Lipschitz continuous near \(\hat{x}_F\) and \(T_{\Omega}^H(\hat{x}_F) \neq \emptyset\). If the set of refining directions is dense for \(\hat{x}_F\) in \(T_\Omega^{Cl}(\hat{x}_F)\), then \(\hat{x}_F\) is a Pareto-Clarke critical point of \((MOP)\).
	\end{theorem}
	\begin{proof}
		The authors in~\cite{AuDe2006} prove than for any direction \(v\) in the Clarke tangent cone,
		\[f^0_i(\hat{x}_F; v) = \lim_{\begin{array}{c} d \in T_{\Omega}^{H}(\hat{x}_F) \\ d \rightarrow v \end{array}} f_i^o(\hat{x}_F; d) \text{ for } i = 1,2, \ldots, m.\]
		By hypothesis, the set of refining directions is dense for \(x_{F} \in \Omega\) in \(T^{Cl}_{\Omega}(\hat{x}_F)\). Then there exists a sequence of refining directions \(\{d_r\}_{r \in R} \in T_{\Omega}^H\) for \(\hat{x}_F\) such that \(\lim_{r \in R} d_r = v\). Since the number of components of the objective function is finite, one can find a subsequence \(\{d_r\}_{r \in R'}\) with \(R' \subseteq R\) such that \(v = \lim_{r \in R'} d_r\) and \(f_{i(v)}(\hat{x}_F; v) \geq 0\) by Theorem~\ref{thm:direction_clarke_stationary} for all indexes \(r \in R'\). Passing at the limit concludes the proof.
	\end{proof}

	\subsection[Infeasible case: results for h]
	{Infeasible case: results for \(h\)}
	
	In this subsection, the goal is to analyse refining subsequences of infeasible points according to the \(h\) violation function as in~\cite{AuDe09a} for the single-objective constrained case. Two cases can occur. The refining point \(\hat{x}_I\) of an infeasible refining subsequence satisfies \(h(x_I) = 0\). In this case, it means that the feasible set is non-empty, and that \(\hat{x}_I\) is a global minimum for the single-objective problem \(\min_{x \in \mathcal{X}} h(x)\). Otherwise, this work proves than \(\hat{x}_I\) satisfies some stationarity results relatively to \(h\). Note that the DMulti-MADS-TEB variant generates an infeasible refining subsequence if and only if it starts from an infeasible point belonging to \(V^0\) and generates no feasible points along the iterations.
	
	Contrary to the feasible case, this work does not characterize particular infeasible sequences of points within the sequence of infeasible frame incumbents. 
	
	\begin{theorem}~\label{thm:direction_clarke_stationary_h}
		Let assumptions~\ref{assumption:starting_points} and~\ref{assumption:boundness_properties} hold and suppose DMulti-MADS generates a refining subsequence \(\{x^k_I\}_{k \in K}\), with \(x^k_I \in I^k\), converging to a refining point \(\hat{x}_I \in X\). Assume that \(h\) is Lipschitz continuous near \(\hat{x}_I\). If \(d \in T^H_{\mathcal{X}}(\hat{x}_I)\) is a refining direction for \(\hat{x}_I\), then \(h^o(\hat{x}_I; d) \geq 0\).
	\end{theorem}
	\begin{proof}
		The proof is similar to that of~\ref{thm:direction_clarke_stationary}, \(h\) and \(\mathcal{X}\) playing respectively the roles of \(f_i\) for a fixed objective index \(i \in \{1,2,\ldots, m\}\) and \(\Omega\).
	\end{proof}

	The next theorem's proof is identical to~\ref{thm:Pareto_Clarke_stationnarity}.

	\begin{theorem}
		Let assumptions~\ref{assumption:starting_points} and~\ref{assumption:boundness_properties} hold and suppose DMulti-MADS generates a refining subsequence \(\{x^k_I\}_{k \in K}\), with \(x^k_I \in I^k\), converging to a refining point \(\hat{x}_I \in \mathcal{X}\). Assume that \(h\) is Lipschitz continuous near \(\hat{x}_I\) and \(T_{\mathcal{X}}^H(\hat{x}_I) \neq \emptyset\). If the set of refining directions is dense in \(T_{\mathcal{X}}^{Cl}(\hat{x}_I)\), then \(\hat{x}_I\) is a Clarke stationary point for
		\[\min_{x \in \mathcal{X}} h(x).\]
	\end{theorem}
	\begin{proof}
		The authors in~\cite{AuDe2006} prove than for any direction \(v\) in the Clarke tangent cone,
		\[h^0(\hat{x}_I; v) = \lim_{\begin{array}{c} d \in T_{\mathcal{X}}^{H}(\hat{x}_I) \\ d \rightarrow v \end{array}} h^o(\hat{x}_I; d).\]
		By hypothesis, the set of refining directions is dense for \(\hat{x}_{I} \in \mathcal{X}\) in \(T^{Cl}_{\mathcal{X}}(\hat{x}_I)\). Then there exists a sequence of refining directions \(\{d_r\}_{r \in R} \in T_{\mathcal{X}}^H\) for \(\hat{x}_I\) such that \(\lim_{r \in R} d_r = v\). By Theorem~\ref{thm:direction_clarke_stationary_h}, for all \(r \in R\), \(h^o(\hat{x}_I; d) \geq 0\). Passing at the limit concludes the proof.
	\end{proof}

	\section{Computational experiments}\label{sect:Numerical_experiments}
	
	This section is devoted to the computational experiments of DMulti-MADS on constrained multiobjective problems. The first part presents the considered solvers. The second part is dedicated to the comparison of all solvers and DMulti-MADS variants on a set of analytical problems using data profiles for multiobjective optimization~\cite{BiLedSa2020}. The last part shows comparison of solvers on ``real'' engineering problems using convergence profiles.
	
	To assess the performance of different algorithms, this work relies on the use of data profiles for multiobjective optimization~\cite{BiLedSa2020} and convergence profiles. Both tools require the definition of a convergence test for a given computational problem, based on the hypervolume indicator~\cite{Zitzler1998}.
	
	The hypervolume indicator represents the volume of the objective space dominated by a Pareto front approximation \(Y_N\) and delimited from above by a reference point \(r \in \mathbb{R}^m\) such that for all \(y \in Y_N\), \(y_i < r_i\) for \(i = 1,2,\ldots,m\). The hypervolume possesses many useful properties : Pareto compliant with the dominance ordering, it can capture many properties of a Pareto front approximation as spread, cardinality, convergence to the Pareto front, or extension~\cite{AuBiCaLedSa2018, Li2019}.
	
	The convergence test requires a Pareto front approximation reference \(Y^p\) for a given problem \(p \in \mathcal{P}\), where \(\mathcal{P}\) is the set of considered problems, from which the approximated ideal objective vector
	\[\tilde{y}^{I,p} = \left(\min_{y \in Y^p} y_1, \min_{y \in Y^p} y_2, \ldots, \min_{y \in Y^p} y_{m_p} \right)^\top\]
	and the approximated nadir objective vector
	\[\tilde{y}^{N,p} = \left(\max_{y \in Y^p} y_1, \max_{y \in Y^p} y_2, \ldots, \max_{y \in Y^p} y_{m_p} \right)^\top\]
	are extracted, with \(m_p\) the number of objectives of problem \(p \in \mathcal{P}\). \(Y^p\) is constructed using the set of best non dominated points found by all algorithms on problem \(p \in \mathcal{P}\) for a maximal budget of evaluations.
	
	Assuming \(Y^e\) is a Pareto front approximation generated after \(e\) evaluations by a given deterministic solver for problem \(p\), a scaling and translating transformation is applied to this last one defined by: \(\forall y \in Y^e \cup Y^p \cup \{\tilde{y}^{N,p}\}\),
	\[T(y) = \begin{cases}
		(y - \tilde{y}^{I,p}) \oslash (\tilde{y}^{N,p} -y) & \text{ if } \tilde{y}^{N,p} \neq \tilde{y}^{I,p}, \\
		y - \tilde{y}^{I,p} & \text{ otherwise;}
	\end{cases}\]
	where \(\oslash\) is the element wise-divisor operator. Note that this transformation conserves the dominance order relation. The computational problem is said to be solved by the algorithm with tolerance \(\varepsilon_\tau > 0\) if
	\[\dfrac{HV\left(T(Y^e), T(\tilde{y}^{N, p})\right)}{HV\left(T(Y^p), T(\tilde{y}^{N, p})\right)} \geq 1 - \varepsilon_\tau\]
	where \(HV(Y_N,r)\) is the hypervolume indicator value of the volume dominated by the Pareto front approximation \(Y_N\) and delimited above by the reference point \(r \in \mathbb{R}^m\). All elements of \(Y_N\) which are dominated by \(r \in \mathbb{R}^m\) are removed during the computation of the hypervolume indicator. If no element of \(Y_N\) dominates \(r\), then \(HV(Y_N, r) = 0\).
	
	Data profiles show the proportion of all computational problems solved by an algorithm in function of the number of groups of  \(n + 1\) evaluations required to build a gradient simplex in \(\mathbb{R}^n\). In these experiments, stochastic solvers are also considered. In this case, data profiles are modified to take into account their performance variability, as described in~\cite{BiLedSa2020}.
	
	\subsection{Tested solvers and variants of DMulti-MADS}
	
	The following constrained solvers are considered:
	\begin{itemize}
		\item the deterministic solver \texttt{NOMAD}~\cite{Le09b} which implements the BiMADS algorithm (Bi-objective Mesh Adaptive Direct Search)~\cite{AuSaZg2008a} tested only for \(m = 2\) objectives - \url{www.gerad.ca/nomad/};
		\item the deterministic solver DFMO (Derivative-Free Multi Objective)~\cite{LiLuRi2016} - \url{http://www.iasi.cnr.it/~liuzzi/DFL/};
		\item the stochastic heuristic solver NSGA-II (Non Dominating Sorting Algorithm II)~\cite{Deb2000}; a constrained version is implemented in the \texttt{Pymoo} Library~\cite{BlankDeb2020} version \texttt{0.4.2.2} - \url{https://pymoo.org}.
	\end{itemize}
	For the BiMADS algorithm, two variants based on \texttt{NOMAD 3.9.1} are considered. The first uses the default settings of the MADS algorithm, detailed in~\cite{AuIaLeDTr2014, AuTr2018, AuLeDTr2018, CoLed2011}. The second deactivates models and other heuristics such that BiMADS relies only on the MADS algorithm with \(n+1\) directions, a speculative search and an opportunistic polling strategy, for a fairer comparison with DMulti-MADS.
	DFMO and NSGA-II are used with their default settings. NSGA-II uses an initial population with \(100\) elements.
	
	In these experiments, this work considers another variant of DMulti-MADS for constrained multiobjective optimization based on the penalty approach used in~\cite{LiLuRi2016}. More specifically, given the constrained multiobjective problem \((MOP)\), the authors of~\cite{LiLuRi2016} introduce the following penalty functions
	\[Z_i(x; \varepsilon) = f_i(x) + \dfrac{1}{\epsilon} \sum_{j \in \mathcal{J}} \max \{0, c_i(x)\}, \ i = 1, 2,\ldots, m\]
	where \(\epsilon > 0\) is an external parameter and consider the following multiobjective problem
	\[(MOP_p) : \min_{x \in \mathcal{X}} Z(x) = \left(Z_1(x; \epsilon), Z_2(x; \epsilon), \ldots, Z_m(x; \epsilon)\right)^\top.\]
	The DMulti-MADS-Penalty variant uses the DMulti-MADS-TEB variant on the modified \((MOP_p)\) multiobjective problem. The external parameter \(\epsilon > 0\) is set to the default value proposed by~\cite{LiLuRi2016}. Note that this approach has already been used by these authors to compare DMS (which cannot start from infeasible points) and DFMO on constrained multiobjective problems~\cite{LiLuRi2016}. As the first strategy proposed to handle constraints with convergence results,
	it is natural to see if this approach performs well compared to the two new variants proposed in this paper.
	
	For all constrained variants of DMulti-MADS, a speculative search strategy is implemented as in~\cite{BiLedSa2020} for one or both feasible and infeasible current incumbents if they exist, combined with a polling strategy with \(n+1\) directions~\cite{AuIaLeDTr2014}. The implementation of the mesh follows a granular mesh strategy~\cite{AuLeDTr2018}. All variants stop as soon as one component of the mesh size vector is below \(10^{-9}\) or after running out of evaluations budget. All variants use an opportunistic strategy: as soon as a new candidate dominates at least one current incumbent, the iteration stops. All variants also apply a spread selection with parameter value \(w^+=1\). For the DMulti-MADS-PB variant, the trigger parameter is set to \(\rho =0.1\). When DMulti-MADS-TEB switches from the first phase to the second phase, the frame and mesh size parameters of the generated feasible points are not resettled to their respective initial values \(\Delta^0\) and \(\delta^0\).
	
	Finally, the implementation of the progressive barrier in \texttt{NOMAD 3.9.1} for the BiMADS algorithm diverges from the description given in~\cite{AuDe09a} for efficiency gains. The DMulti-MADS-PB algorithm variant equally incorporates these modifications for a fairer comparison with the implementation of \texttt{NOMAD 3.9.1}. Precisely, the threshold \(h^k_{\max}\) is updated according to the set \(U^{k+1} \subseteq V^{k+1}\), which enables it to decrease faster. Furthermore, in the implementation, an iteration \(k\) is considered as improving if the algorithm generates a point \(x^t \in V^{k+1} \setminus V^k\) satisfying improving conditions. Note that the convergence properties for infeasible refining subsequences still hold. However, it may exist a point \(x \in \cup_{k \in \mathbb{N}} V^k\) with \(0 < h(x) < h(\hat{x}^I)\) where \(\hat{x}^I \in \mathcal{X}\) is an infeasible refining point.
	
	The code used for experiments can be found at \url{https://github.com/bbopt/DMultiMadsPB}.
	
	\subsection{Comparing solvers on synthetic benchmarks}
	
	In this subsection, this work considers a set of \(214\) analytical multiobjective optimization problems proposed by~\cite{LiLuRi2016}, with \(n \in [3,30]\), \(m \in \{2,3,4\}\) and \(|\mathcal{J}| \in [3,30]\). Among them, \(103\) problem possess \(m = 2\) objectives.
	
	In a first part, this work compares the three variants of DMulti-MADS on this set of problems. The three of them use a maximum budget of \(30,000\) evaluations. For each problem, the three variants start from the same set of initial points, using the linesearch strategy described in~\cite{CuMaVaVi2010}. Each variant on each problem executes \(10\) replications by changing the random seed which controls the generation of polling directions.
	
	\begin{figure}[!th]
		\centering
		\subfigure[\(\varepsilon_{\tau}=10^{-2}\)]{
			\includegraphics[scale=0.6]{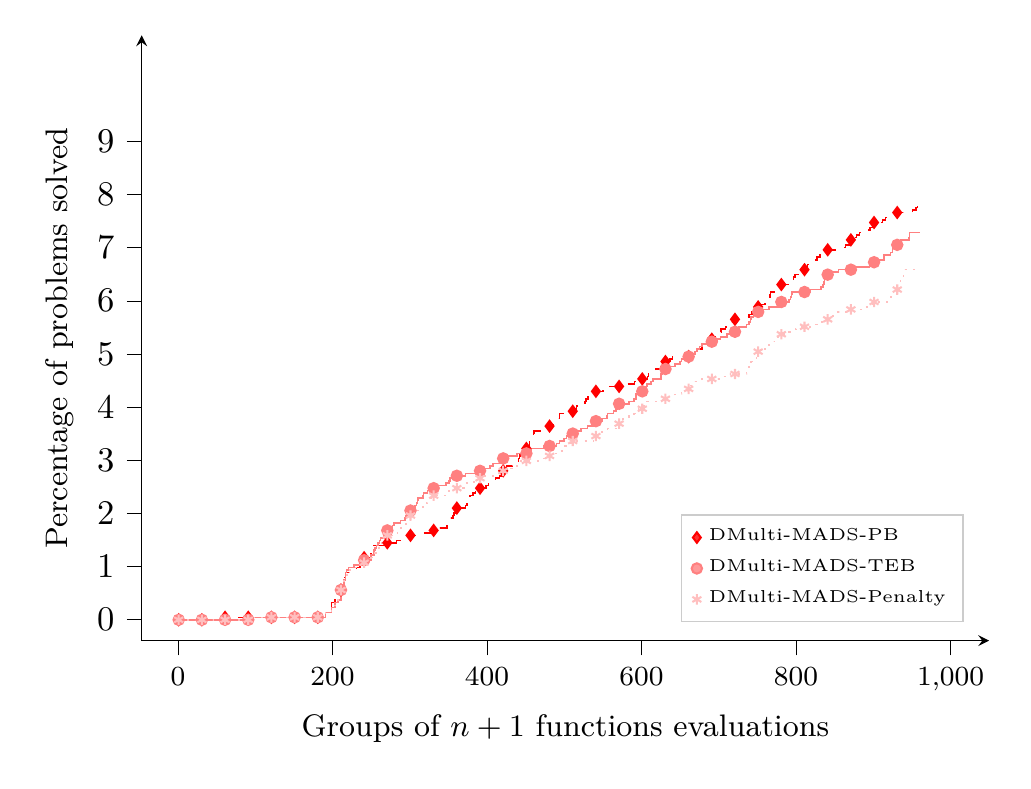}}
		\quad
		\subfigure[\(\varepsilon_{\tau}=5 \times 10^{-2}\)]{
			\includegraphics[scale=0.6]{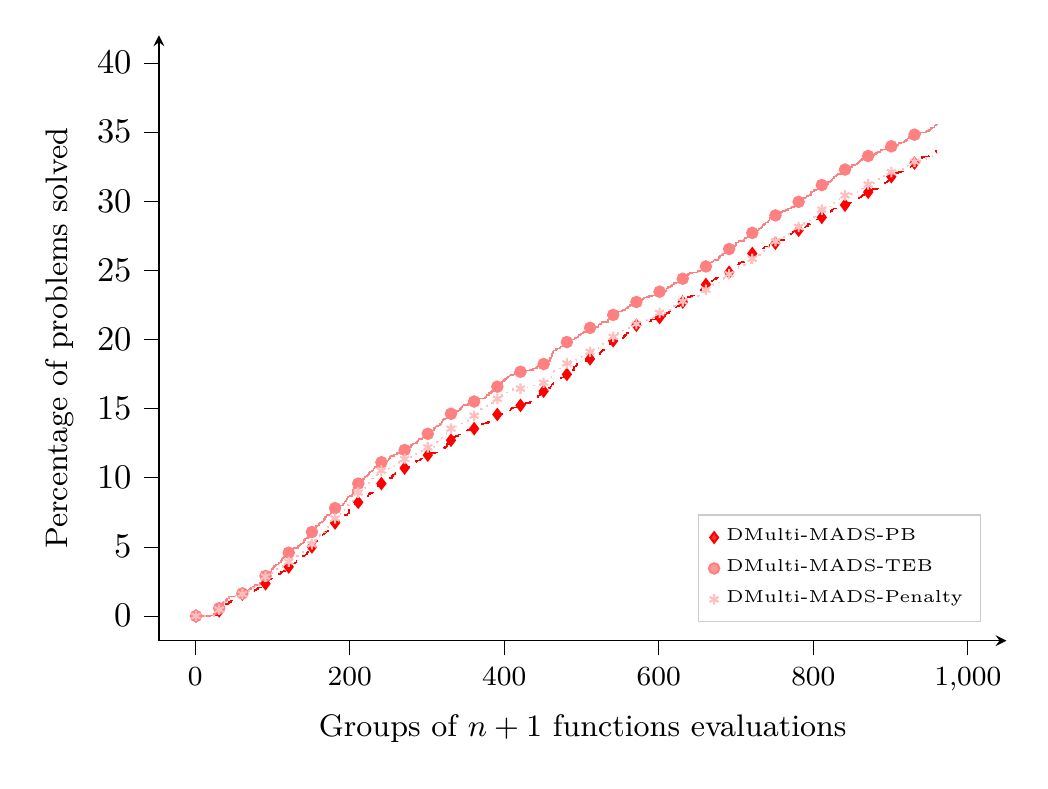}}
		\quad
		\subfigure[\(\varepsilon_{\tau}=10^{-1}\)]{
			\includegraphics[scale=0.6]{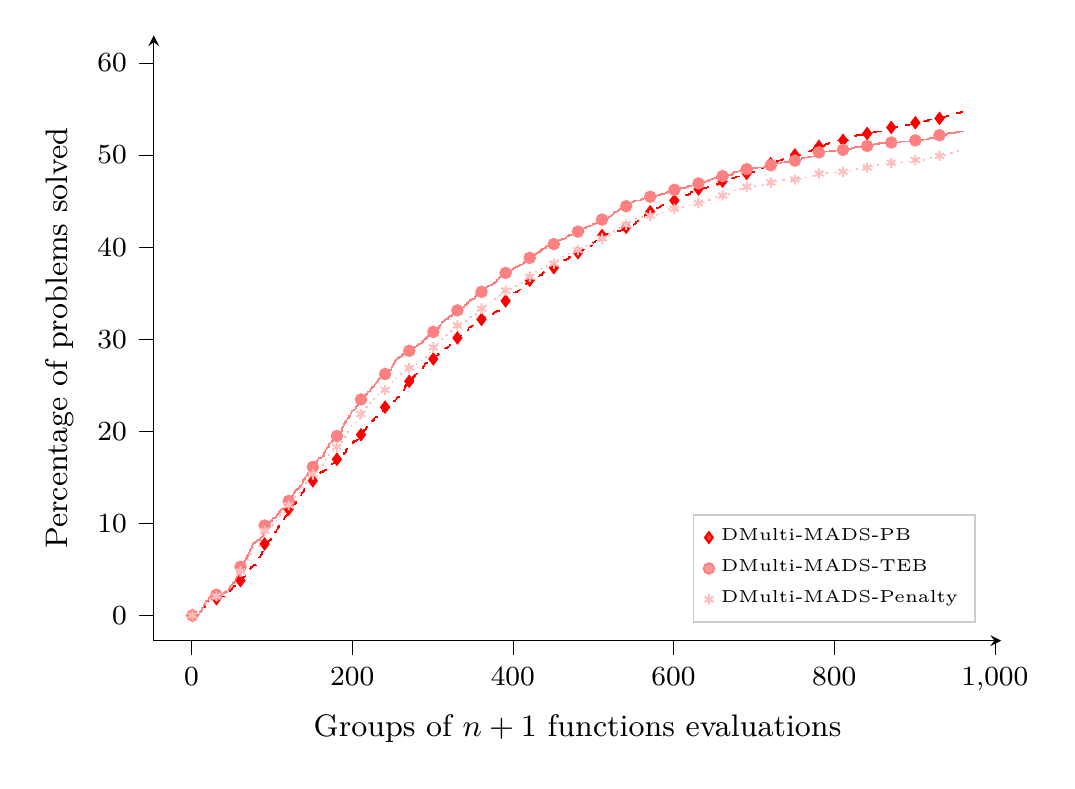}}
		\caption{Data profiles obtained from \(10\) replications from \(214\) multiobjective analytical problems taken from~\cite{LiLuRi2016} for DMulti-MADS-PB, DMulti-MADS-TEB and DMulti-MADS-Penalty with tolerance \(\varepsilon_{\tau} \in \{10^{-2}, 5 \times 10^{-2}, 10^{-1}\}\).}
		\label{fig:dmultimads_variants_synthetic_benchmarks}
	\end{figure}

	The data profiles given in Figure~\ref{fig:dmultimads_variants_synthetic_benchmarks} show that for the three tolerance values considered, DMulti-MADS-Penalty solves slightly less problems than the two other variants introduced in this work. One can equally observe than DMulti-MADS-PB performs better for a medium to high budget of evaluations for the lowest tolerance \(\varepsilon_{\tau} = 10^{-2}\). For the largest tolerance \(\varepsilon_{\tau} = 10^{-1}\), DMulti-MADS-PB solves more problems for a high budget of evaluations. However, for medium tolerance, the performance of DMulti-MADS-PB is similar to DMulti-MADS-Penalty. A closer look at the considered problems shows than in this case, it is better to firstly look for feasible solutions than to explore the infeasible decision space. It then gives an advantage to DMulti-MADS-TEB over the two other variants. For the rest of this subsection, only DMulti-MADS-PB and DMulti-MADS-TEB are kept, as they are more performant.
	
	For the comparison with the other algorithms, the same maximum budget of \(30,000\) function evaluations is kept. Practically, for NSGA-II, the total number of population generations is fixed to \(300\), with a fixed population size equal to \(100\). For each problem, the deterministic solvers start from the same initial points using the linesearch strategy~\cite{CuMaVaVi2010}. For each problem, NSGA-II is run \(30\) times with different seeds to capture stochastic behavior and analyze its performance variation.
	
	\begin{figure}[!th]
		\centering
		\subfigure[\(\varepsilon_{\tau}=10^{-2}\)]{
			\includegraphics[scale=0.6]{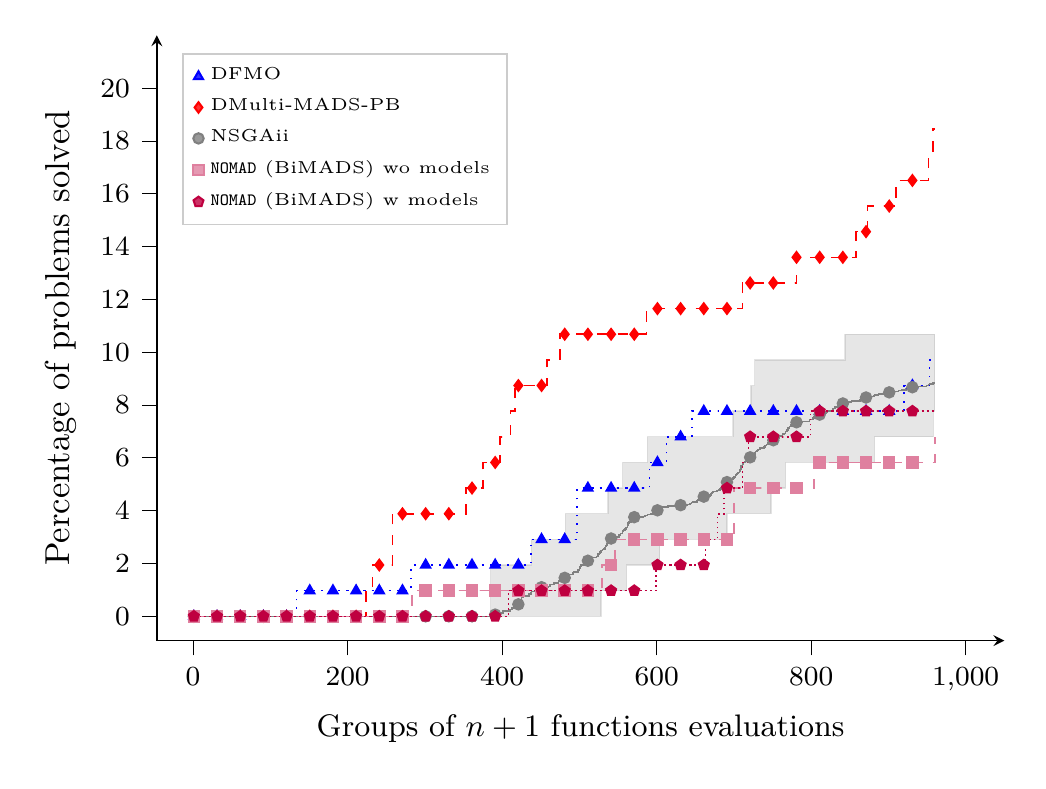}}
		\quad
		\subfigure[\(\varepsilon_{\tau}=5 \times 10^{-2}\)]{
			\includegraphics[scale=0.6]{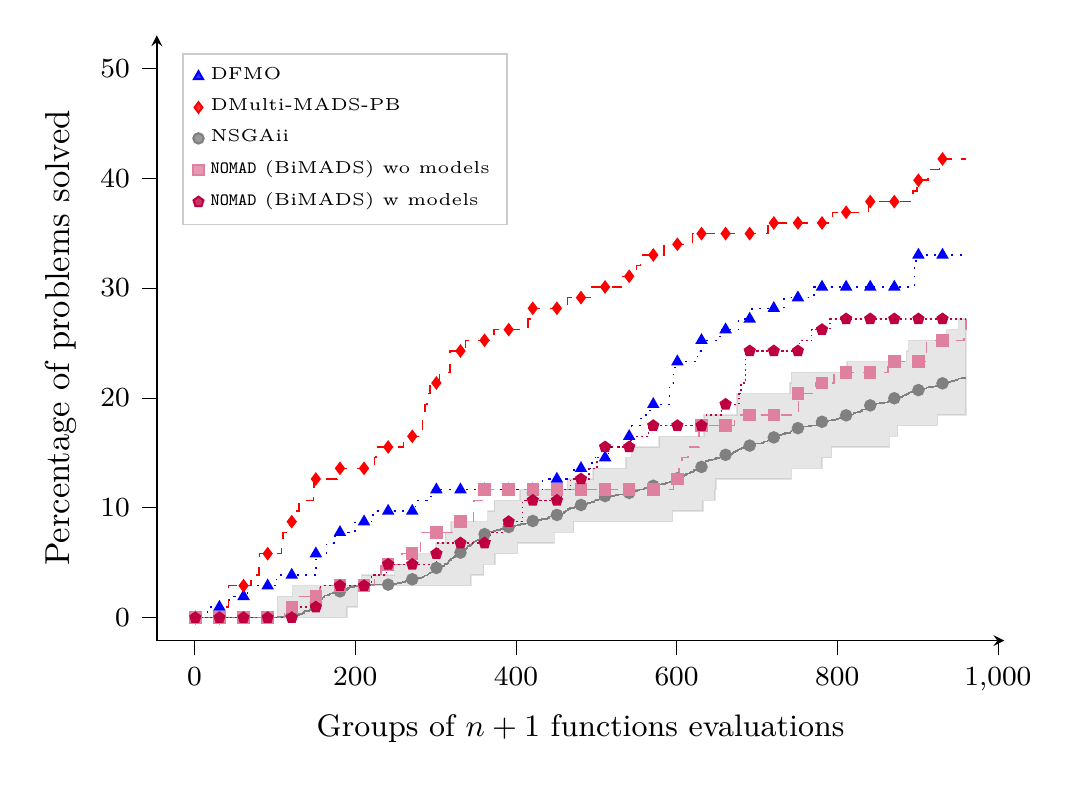}}
		\quad
		\subfigure[\(\varepsilon_{\tau}=10^{-1}\)]{
			\includegraphics[scale=0.6]{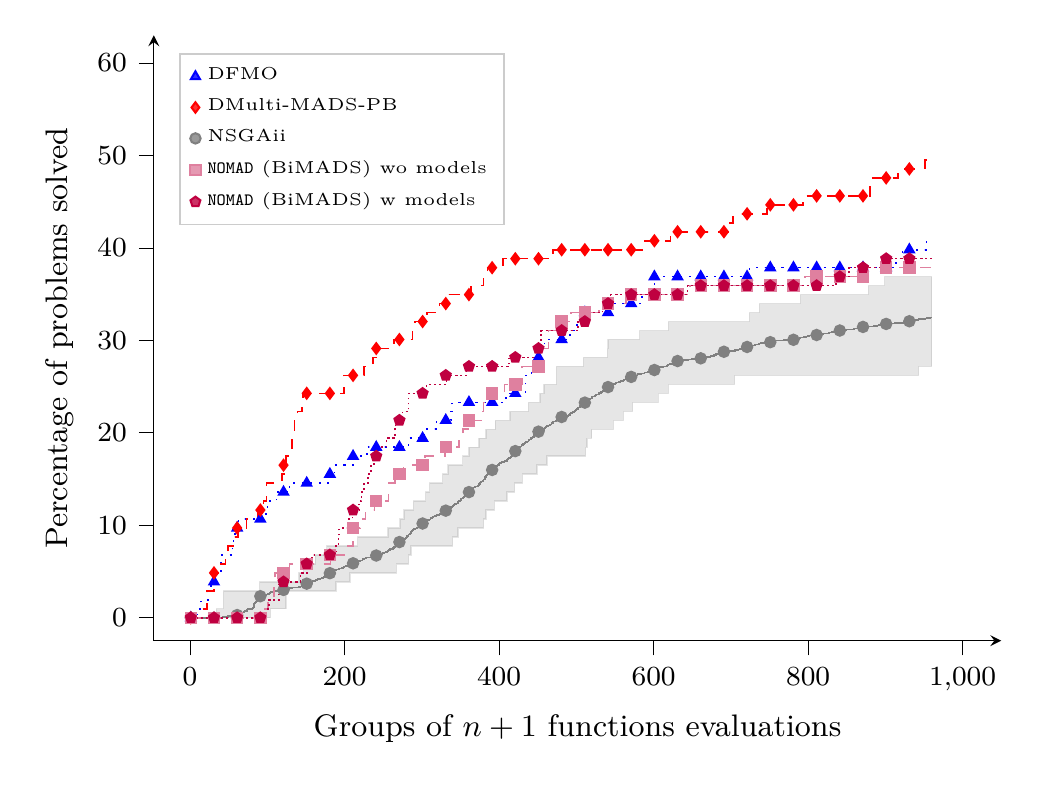}}
		\caption{Data profiles using \texttt{NOMAD} (BiMADS), DFMO, DMulti-MADS-PB and NSGA-II obtained on \(103\) biojective analytical problems from~\cite{LiLuRi2016} with \(30\) different runs of NSGA-II with tolerance \(\varepsilon_{\tau} \in \{10^{-2}, 5 \times 10^{-2}, 10^{-1}\}\).}
		\label{fig:PB_vs_others_2obj_synthetic_benchmarks}
	\end{figure}

	\begin{figure}[!th]
		\centering
		\subfigure[\(\varepsilon_{\tau}=10^{-2}\)]{
			\includegraphics[scale=0.6]{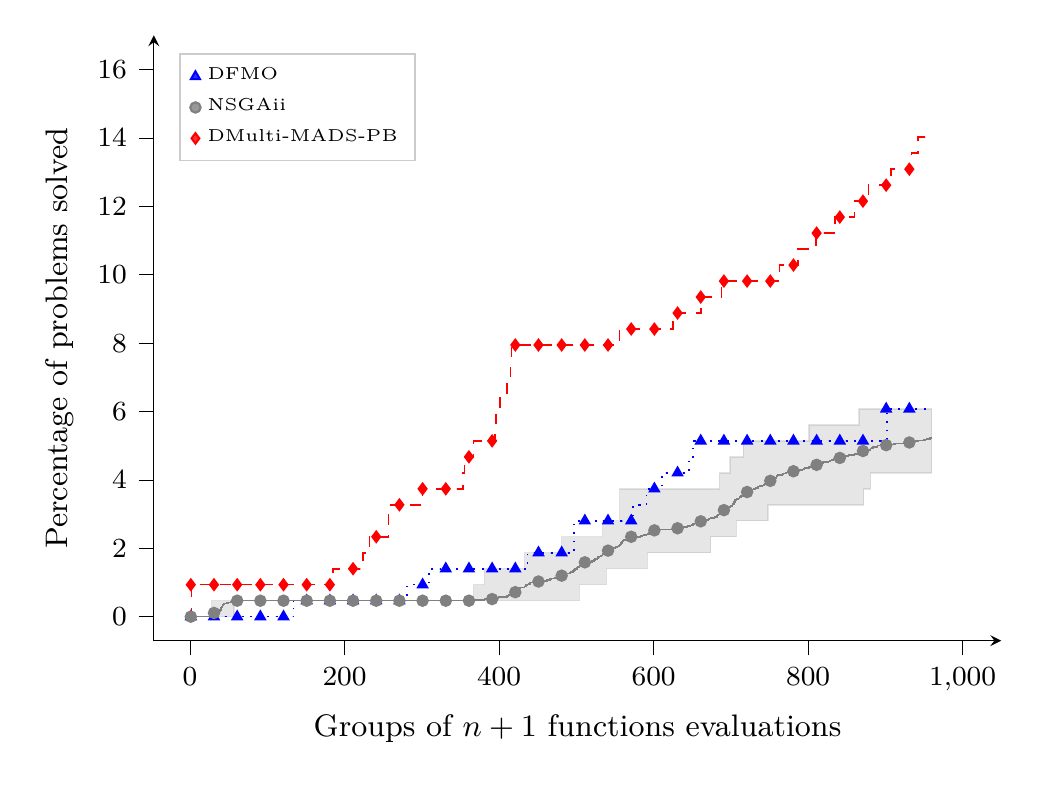}}
		\quad
		\subfigure[\(\varepsilon_{\tau}=5 \times 10^{-2}\)]{
			\includegraphics[scale=0.6]{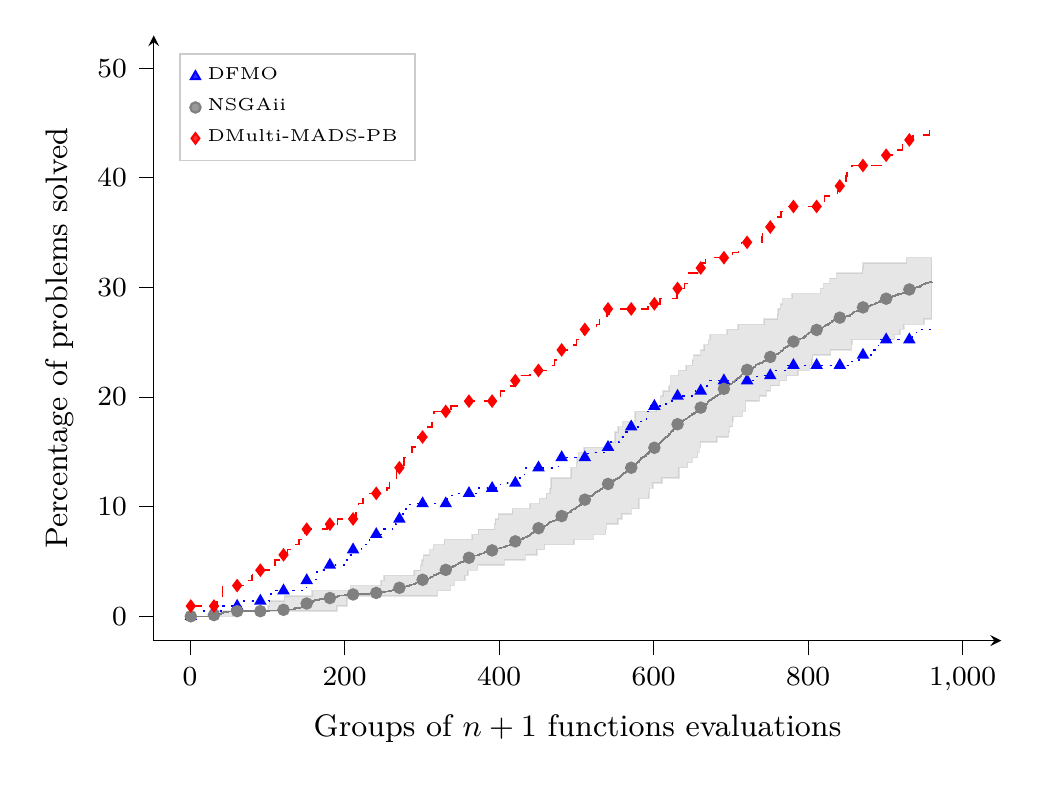}}
		\quad
		\subfigure[\(\varepsilon_{\tau}=10^{-1}\)]{
			\includegraphics[scale=0.6]{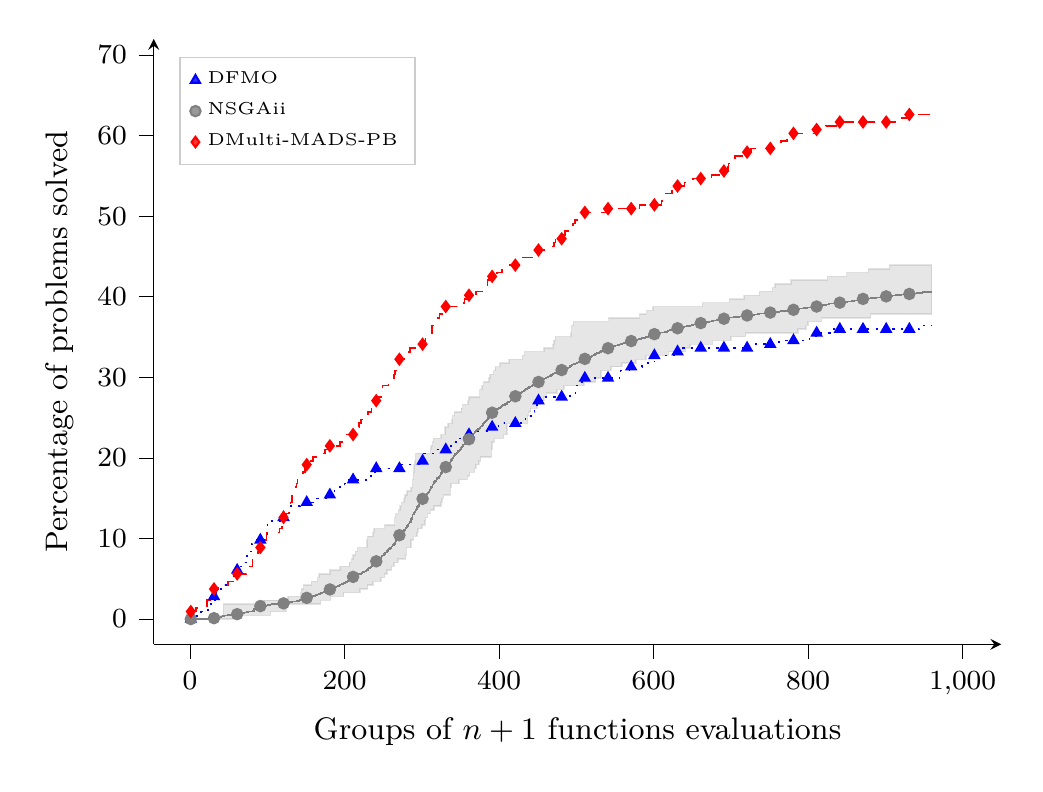}}
		\caption{Data profiles using DFMO, DMulti-MADS-PB and NSGA-II obtained on \(214\) multiobjective analytical problems from~\cite{LiLuRi2016} with \(30\) different runs of NSGA-II with tolerance \(\varepsilon_{\tau} \in \{10^{-2}, 5 \times 10^{-2}, 10^{-1}\}\).}
		\label{fig:PB_vs_others_3obj_synthetic_benchmarks}
	\end{figure}

	From Figures~\ref{fig:PB_vs_others_2obj_synthetic_benchmarks} and~\ref{fig:PB_vs_others_3obj_synthetic_benchmarks}, one can see that DMulti-MADS-PB outperforms the other solvers on this set of analytical functions, for all tolerances considered. The same conclusions can be drawn for DMulti-MADS-TEB. From these figures, one can also observe than DFMO displays better performance on biojective problems than for the whole set (see Figure~\ref{fig:PB_vs_others_2obj_synthetic_benchmarks}).

	\subsection{Comparing solvers on real engineering benchmarks}
	
	In this subsection, this work considers three multiobjective optimization problems: the biobjective
	SOLAR8 and SOLAR9 design problems and the triobjective STYRENE design problem~\cite{AuBeLe08, AuSaZg2010a}. These three applications are more costly to solve than the analytical benchmarks considered in the previous subsection. The use of data profiles to compare solvers on these problems is then difficult to put into practice.
	
	To assess the performance of solvers on these problems, an adaptation of convergence profiles (see~\cite[Appendix \(A\)]{AuHa2017} for a description) to multiobjective optimization is proposed. Convergence profiles for multiobjective optimization make use of the normalized hypervolume value, presented at the beginning of Section~\ref{sect:Numerical_experiments} and given by:
	\[\dfrac{HV\left(T(Y^e), T(\tilde{y}^{N,p})\right)}{HV(T(Y^p), T(\tilde{y}^{N,p}))}\]
	where \(Y^p\) is the Pareto front approximation reference for problem \(p\), \(Y^e\) the Pareto front approximation generated after \(e\) evaluations by a given solver on an instance of problem \(p\), \(T\) a scaling and translating transformation applied and \(\tilde{y}^{N,p}\) the approximated nadir objective vector of \(Y^p\).
	
	Convergence profiles for multiobjective optimization on a given problem \(p\) visualize the evolution of the normalized hypervolume indicator for a given solver against the number of evaluations used. Consequently, a normalized hypervolume value equal to \(1\) means that the solver has solved the problem \(p\). A normalized hypervolume equal to \(0\) means that the solver has not generated points which dominate the approximated nadir objective vector of the Pareto front approximation reference.
	
	\subsubsection{Comparing solvers on the SOLAR8 and SOLAR9 design problems}
	
	SOLAR8 and SOLAR9 are two biobjective optimization problems derived from a numerical simulator coded in C++ of a solar plant with a molten salt heat storage system~\cite{MScMLG}. The simulation is composed of three steps. The heliostats field captures sun rays which are transmitted to a central cavity receiver. The sun energy is given to the thermal storage which converts it to thermal energy. This last one activates the powerblock, which triggers a steam turbine, generating electrical power output. Numerical simulations intervene all along the different phases of the process, which make it impossible to provide gradients. For more details, the reader can refer to~\cite{MScMLG}. The simulator can be found at \url{https://github.com/bbopt/solar}.
	
	For the two considered problems, a blackbox evaluation can take more than \(10\) seconds (on a machine with \(8\) Intel(R) Core(TM) i7-2600K CPU @ 3.40GHz 16G RAM). Experiments equally reveal the presence of hidden constraints. Tables~\ref{tab:Solar_8_characteristics} and~\ref{tab:Solar_9_characteristics} describe the objectives, constraints and starting points used for each problem.
	
	\begin{figure}[!th]
		\centering \small
		\begin{tabular}{ll}
			\hline
			Constraints/Objectives & Description of constraints and objectives \\
			\hline
			\(-f_1\) & Maximize heliostat field performance (absorbed energy) \\
			\(f_2\) & Minimize cost of field, tower and receiver  \\
			Heliostat design constraints & Four constraints related to the dimensions of the heliostat field \\
			Receiver constraints & Three constraints related to the design of the receiver \\
			Energy constraints & Two constraints which depend on the energy production\\		
			\hline
			Variables & Description and type \\
			\hline
			Heliostats field & Nine variables related to the dimensions of the heliostats field \\
			& Eight real and one integer \\
			Heat transfer loop & Four variables related to the design of the heat transfer system \\
			& Three real and one integer \\
			\hline
			Starting point (infeasible) & \((11.0, 11.0, 200.0, 10.0, 10.0, 2650, 89.0, 0.5, 8.0, 36, 0.30, 0.020, 0.0216)\) \\
			\hline
		\end{tabular}
		\caption{Objectives, constraints, variables and starting point of the SOLAR8 problem.}
		\label{tab:Solar_8_characteristics}
	\end{figure}
	
	\begin{figure}[!th]
		\centering\small
		\begin{tabular}{ll}
			\hline
		 	Constraints/Objectives & Description of constraints and objectives \\
		 	\hline
		 	\(f_1\) &  Minimize production costs \\
		 	\(-f_2\) & Maximize energy production \\
		 	Heliostats design constraints & Four constraints related to the dimensions of the heliostat field \\
		 	Heat storage constraints & Four constraints relative to the molten salt heat thermic/pressure \\
		 	& storage system \\
		 	Receiver design constraints & Two constraints which depend on the tube size and diameter receiver \\
		 	Steam constraints & Five constraints related to steam temperature, power output \\
		 	& and steam design. \\
		 	\hline
		 	Variables & Description and type \\
		 	\hline
		 	Heliostats field & Nine variables related to the dimensions of the heliostats field \\
		 	& Eight real and one integer \\
		 	Heat transfer loop & Nineteen variables related to the design of the heat transfer system \\
		 	& Fourteen real and five integer \\
		 	Powerblock & One variable: type of turbine; integer \\
		 	\hline
		 	Starting point (infeasible) & \((9.0, 9.0, 150.0, 6.0, 8.0, 1000, 45.0, 0.5, 5.0, 900.0,\) \\
		 	& \(9.0, 9.0, 0.30, 0.20,560.0, 500, 0.30, 0.0165, 0.018, 0.017,\) \\
		 	& \(10.0, 0.0155, 0.016, 0.20, 3, 12000, 1, 2, 2)\) \\
		 	\hline
		\end{tabular}
		\caption{Objectives, constraints, variables and starting point of the SOLAR9 problem.}
		\label{tab:Solar_9_characteristics}
	\end{figure}

	SOLAR8 and SOLAR9 both possess integer decision variables. In the experiments, the only solver which can treat integer variables is \texttt{NOMAD} (BiMADS). Consequently, for the other solvers, all integer variables are fixed to their starting values along the optimization. For the SOLAR8 problem, three variants of \texttt{NOMAD} (BiMADS): two for which integer variables are fixed and one which treat mixed integer (MI) problems. For this last variant, the algorithmic parameters are chosen by default.
	
	\begin{remark}
		For SOLAR9, \texttt{NOMAD} (BiMADS) completely outperforms the other algorithms when it can modify integer variables. After investigation, this behaviour is not related to the performance of the algorithm, but the initial choice of the integer variables. However, for the sake of reproducibility, these values are kept. 
	\end{remark}

	All deterministic algorithms are allocated a maximal budget of \(5,000\) evaluations and start from the same infeasible point for each problem. NSGA-II does not take starting points as arguments. To compare it with the others, NSGA-II is run \(10\) times to capture stochastic behaviour, with a population size fixed to \(100\) and a total number of generations equal to \(50\).
	
	\begin{figure}[!th]
		\centering
		\subfigure{\includegraphics[scale=0.7]{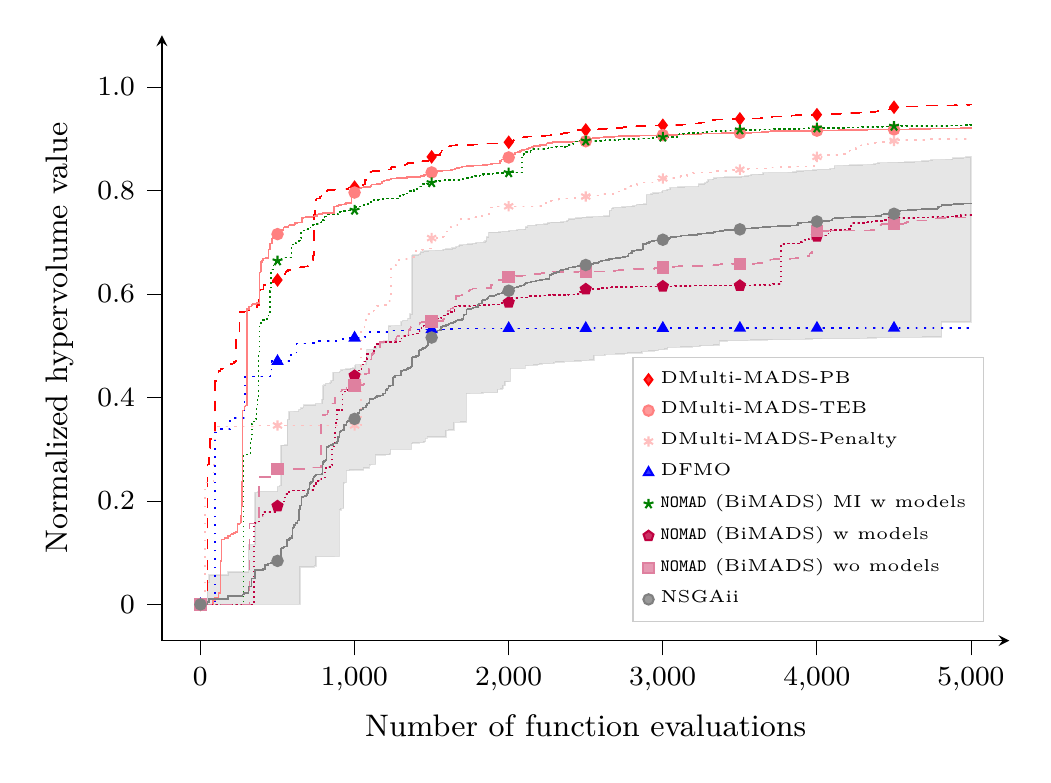}
		\label{fig:solar8_convergence_profiles}}
		\quad
		\subfigure{\includegraphics[scale=0.7]{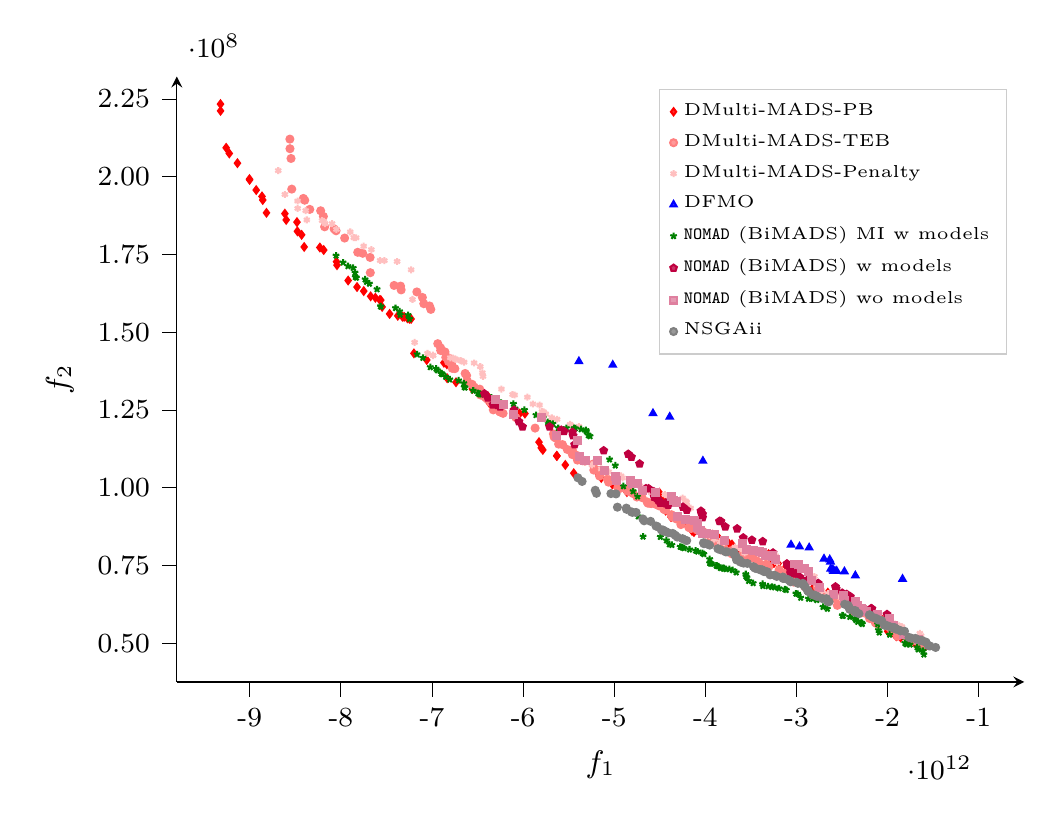}
		\label{fig:solar8_pareto_fronts}}
		\caption{(a) On the left, convergence profiles for the SOLAR8 problem using DFMO, DMulti-MADS, \texttt{NOMAD} (BiMADS) and NSGA-II with \(10\) different runs of NSGA-II for a maximal budget of \(5,000\) evaluations. (b) On the right, Pareto front approximations obtained at the end of the resolution of SOLAR8 for DFMO, DMulti-MADS, \texttt{NOMAD} (BiMADS) and an instance of NSGA-II in the objective space.}
	\end{figure}

	From Figure~\ref{fig:solar8_convergence_profiles}, one can see that DMulti-MADS-PB performs better than the other algorithms on SOLAR8. When looking at the Pareto front plottings (Figure~\ref{fig:solar8_pareto_fronts}), one can note that DMulti-MADS-PB captures a portion of the Pareto front on the top left. DMulti-MADS-TEB is slightly better than DMulti-MADS-Penalty and compares well in terms of performance with \texttt{NOMAD} (BiMADS) when allowing the use of mixed integer variables. DFMO does not perform well on this problem, due to the different scales on the constraints included in the penalty objective function, which impacts its efficiency.

	\begin{figure}[!th]
	\centering
	\subfigure{\includegraphics[scale=0.7]{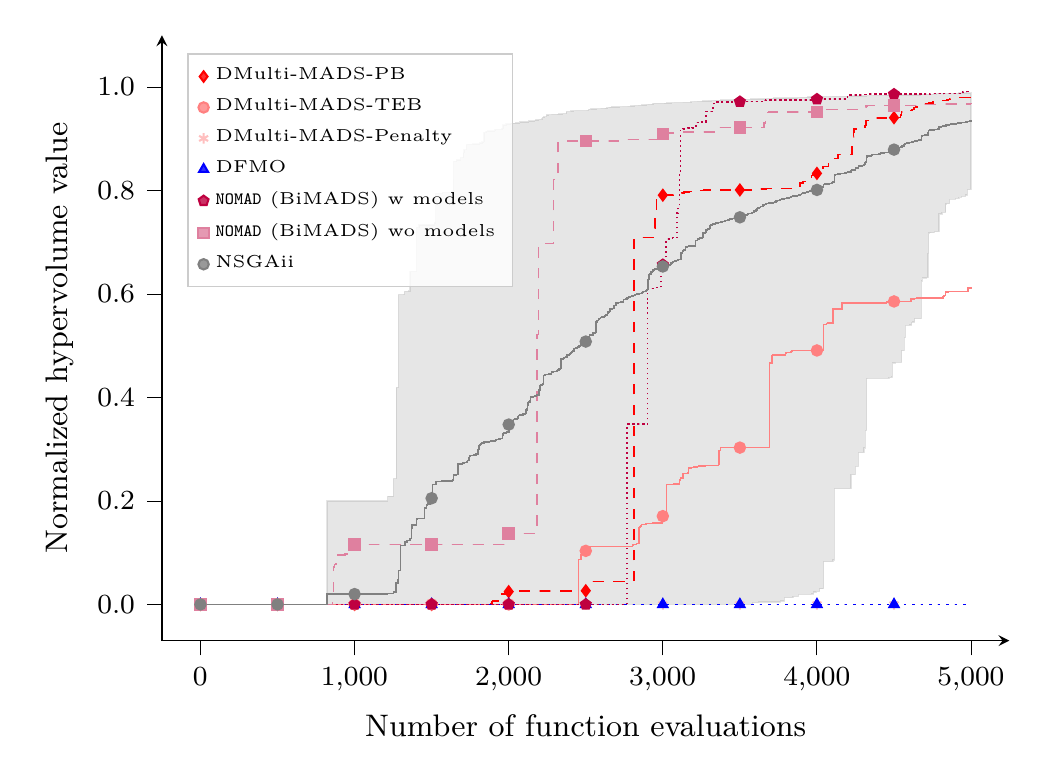}
		\label{fig:solar9_convergence_profiles}}
	\quad
	\subfigure{\includegraphics[scale=0.7]{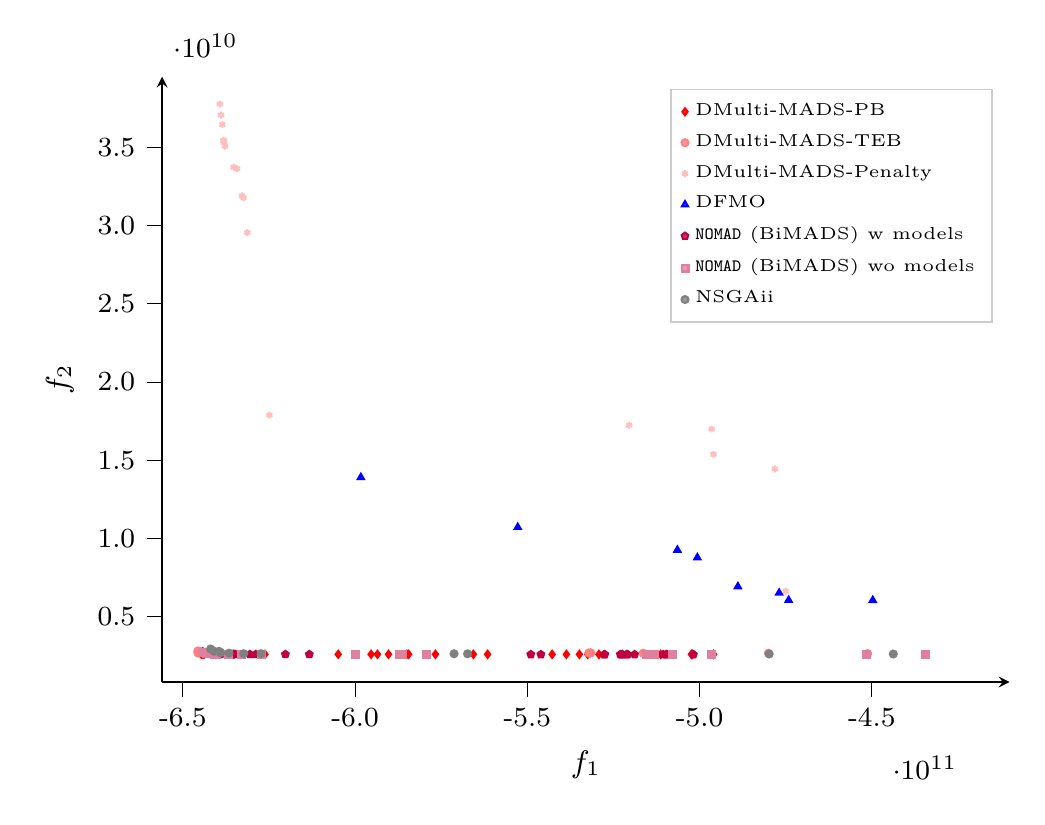}
		\label{fig:solar9_pareto_fronts}}
	\caption{(a) On the left, convergence profiles for the SOLAR9 problem using DFMO, DMulti-MADS, \texttt{NOMAD} (BiMADS) and NSGA-II with \(10\) different runs of NSGA-II for a maximal budget of \(5,000\) evaluations. (b) On the right, Pareto front approximations obtained at the end of the resolution of SOLAR9 for DFMO, DMulti-MADS, \texttt{NOMAD} (BiMADS) and an instance of NSGA-II in the objective space.}
	\end{figure}

	Figure~\ref{fig:solar9_convergence_profiles} shows the convergence profiles for the SOLAR9 problem for different solvers. On this problem, \texttt{NOMAD} (BiMADS) are the most efficient, even if DMulti-MADS-PB catches it for the last evaluations. As shown on Figure~\ref{fig:solar9_pareto_fronts}, the extent of the Pareto front approximation reference is low, which favours scalarization-based approaches such as BiMADS. This problem also illustrates the default of penalty-based approaches against other methods. As the constraint functions possess different amplitudes, the penalized optimization problem differs from the original, which explains why DFMO and DMulti-MADS-Penalty fail.

	\subsubsection{Comparing solvers on the STYRENE design problem}
	
	STYRENE is a triobjective optimization problem related to the production of styrene, as described in~\cite{AuBeLe08, AuSaZg2010a}. Styrene production process is composed of four steps: reactants preparation, catalytic reactions, a first distillation to recover styrene and a second one to recover benzene. The second distillation equally involves the recycling of unreacted ethylbenzaline, reintroduced into the styrene production as an initial reactant. The proposed triobjective optimization problem, based on a numerical implementation coded in C++, aims at maximizing the net present value associated to the process (\(f_1\)), the purity of produced styrene (\(f_2\)), and the overall ethylbenzene conversion into styrene (\(f_3\)). This application possesses eight bounded variables, and nine general inequality constraints related to the chemical process (e.g. environmental regulations), or costs (e.g. investment). More details can be found in~\cite{AuSaZg2010a}.
	
	A simulation takes around \(1\) second to run, starting from a feasible point (on a machine with \(8\) Intel(R) Core(TM) i7-2600K CPU @ 3.40GHz 16G RAM). This problem has hidden constraints. Even when starting from a feasible point, the simulation can sometimes fail to produce a finite numerical value.
	
	A maximal budget of \(20,000\) evaluations is allocated for all deterministic solvers, which all start from the same feasible point. This experiment does not consider \texttt{NOMAD} (BiMADS), as it only treats biojective problems. NSGA-II is run \(10\) times, with a population size fixed to \(100\), and a maximal number of generations equal to \(200\).
	
	\begin{figure}[!th]
		\centering
		\subfigure{\includegraphics[scale=0.7]{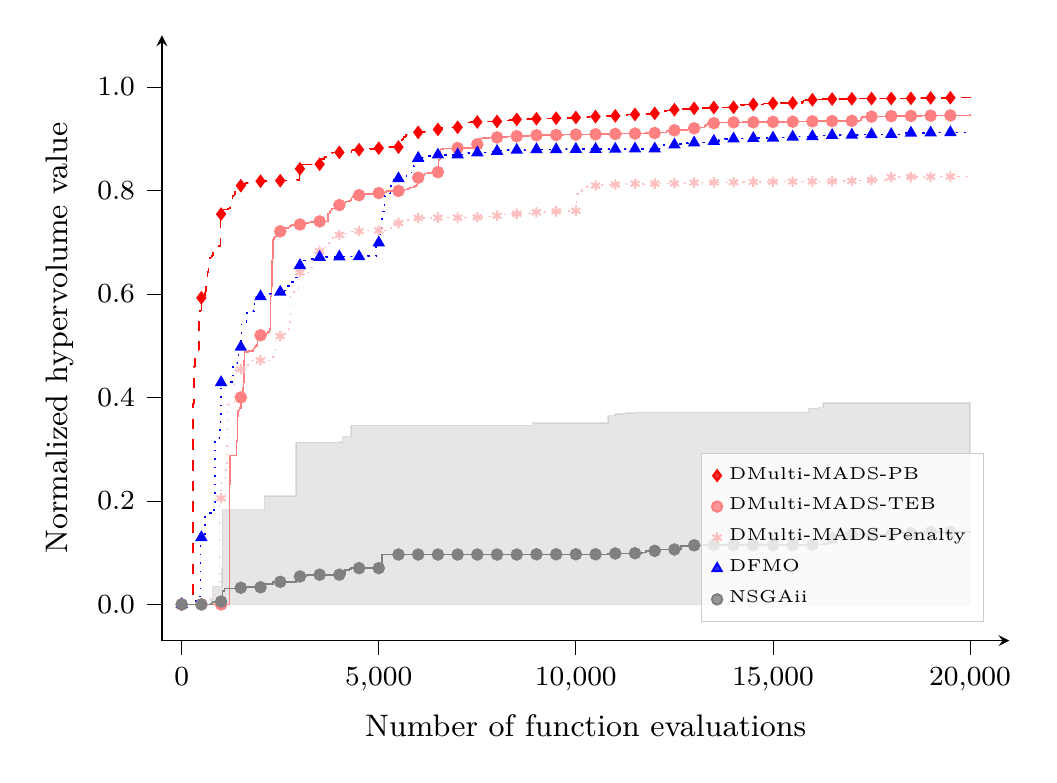}
			\label{fig:styrene_convergence_profiles}}
		\subfigure{\includegraphics[scale=0.7]{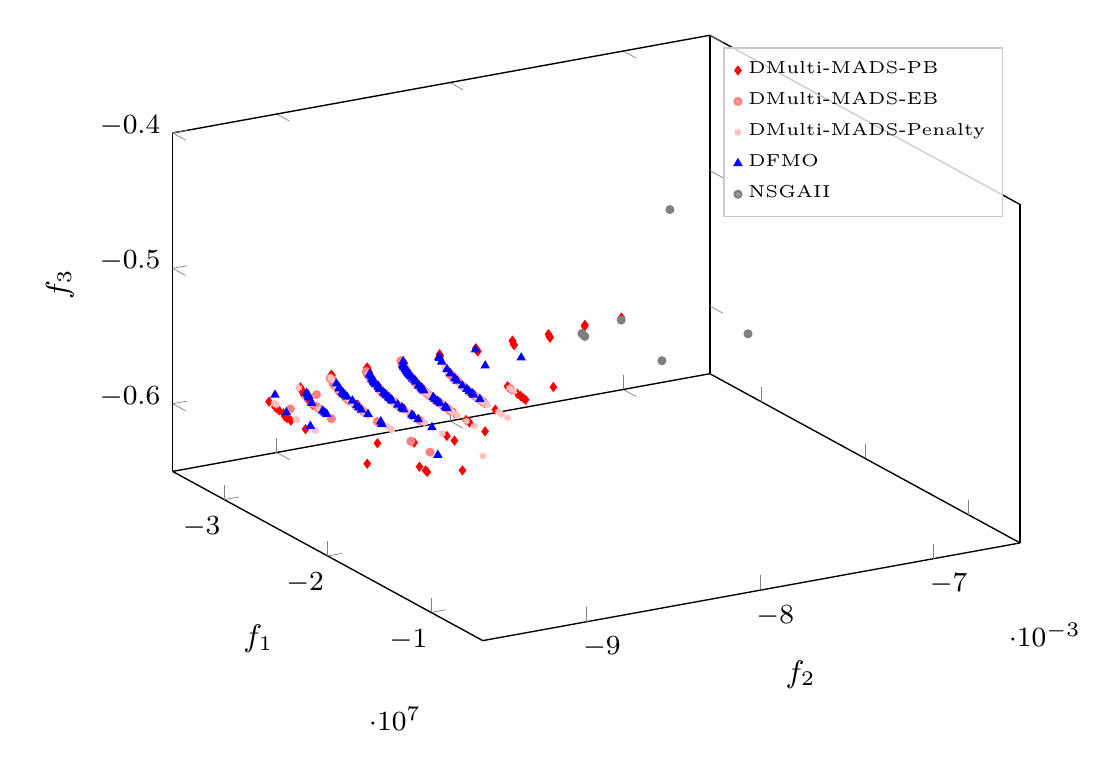}
			\label{fig:styrene_pareto_fronts}}
		\caption{(a) On the left, convergence profiles for the STYRENE design problem using DFMO, DMulti-MADS, and NSGA-II with \(10\) different runs of NSGA-II for a maximal budget of \(20,000\) evaluations. (b) On the right, Pareto front approximations obtained at the end of the resolution of STYRENE for DFMO, DMulti-MADS, and an instance of NSGA-II in the objective space.}
	\end{figure}

	Figure~\ref{fig:styrene_convergence_profiles} shows the convergence profiles obtained for the STYRENE design triobjective problem. This figure shows that DMulti-MADS-PB performs better than the other solvers, followed by DMulti-MADS-TEB.  From Figure~\ref{fig:styrene_pareto_fronts}, one can observe that DMulti-MADS-PB captures more parts of the Pareto front reference than all the other methods. Finally, even when taking into account variability, NSGA-II is less efficient than all the other solvers on this problem.

\section{Discussion}

This work proposes two extensions of the DMulti-MADS algorithm~\cite{BiLedSa2020} to handle blackbox constraints, generalizing the works conducted for the single-objective MADS algorithm~\cite{AuDe09a, AuDeLe10}. It is proved that these two extensions possess the same convergence properties than DMulti-MADS~\cite{BiLedSa2020} when studying feasible sequences generated by these two extensions. Convergence analysis for the infeasible case is also derived, as in~\cite{AuDe09a}.

Experiments show that these two variants are competitive comparing to other state-of-the-art methods, and more robust on real engineering applications than a penalty-based approach, as proposed in~\cite{LiLuRi2016}. These experiments also reveal that a two-phase approach performs surprisingly well on blackbox multiobjective optimization problems, contrary to single-objective ones~\cite{AuDeLe10}.
Future work involves the integration of surrogate methods into a search strategy~\cite{BraCu2020, CoLed2011}, and the use of parallelism.
 An integration in the \texttt{NOMAD} solver is also planned.

\bibliographystyle{plain}
\bibliography{bibliography}

\end{document}